\documentclass[pdflatex,sn-mathphys-num]{sn-jnl}

\usepackage[T1]{fontenc}
\usepackage{lmodern}

\usepackage{amsmath,amssymb,amsfonts}
\usepackage{amsthm}
\usepackage{enumitem}
\usepackage{tikz}
\usepackage{xcolor}

\usepackage[capitalize,nameinlink]{cleveref}
\colorlet{facelight}{gray!12}
\colorlet{facetshade}{gray!55}

\crefname{theorem}{Theorem}{Theorems}
\crefname{lemma}{Lemma}{Lemmas}
\crefname{proposition}{Proposition}{Propositions}
\crefname{corollary}{Corollary}{Corollaries}
\crefname{observation}{Observation}{Observations}
\crefname{remark}{Remark}{Remarks}
\crefname{example}{Example}{Examples}
\crefname{section}{Section}{Sections}
\crefname{subsection}{Subsection}{Subsections}
\crefname{equation}{}{}
\creflabelformat{equation}{(#2#1#3)}

\makeatletter
\newtheoremstyle{thmstyleone}
{18pt plus2pt minus1pt}
{18pt plus2pt minus1pt}
{\small\itshape}
{0pt}
{\small\bfseries}
{}
{.5em}
{\thmname{#1}\thmnumber{\@ifnotempty{#1}{ }\@upn{#2}}%
  \thmnote{ {\the\thm@notefont(#3)}}}
\newtheoremstyle{thmstyletwo}
{18pt plus2pt minus1pt}
{18pt plus2pt minus1pt}
{\small\normalfont}
{0pt}
{\small\itshape}
{}
{.5em}
{\thmname{#1}\thmnumber{\@ifnotempty{#1}{ }{#2}}%
  \thmnote{ {\the\thm@notefont(#3)}}}
\makeatother

\theoremstyle{thmstyleone}%
\newtheorem{theorem}{Theorem}[section]
\newtheorem{lemma}[theorem]{Lemma}
\newtheorem{proposition}[theorem]{Proposition}
\newtheorem{corollary}[theorem]{Corollary}

\theoremstyle{thmstyletwo}%
\newtheorem{remark}[theorem]{Remark}
\newtheorem{example}[theorem]{Example}

\numberwithin{equation}{section}

\hypersetup{bookmarksdepth=2}

\makeatletter
\renewcommand*{\top}{%
  {\mathpalette\@transpose{}}%
}
\newcommand*{\@transpose}[2]{%
  \raisebox{\depth}{$\m@th#1\scriptscriptstyle\mathsf{T}$}%
}
\makeatother

\newcommand{\Po}{\mathcal{P}}
\newcommand{\OP}{\mathcal{O}(P)}
\newcommand{\D}{\mathcal{D}}
\DeclareMathOperator{\conv}{conv}
\DeclareMathOperator{\rank}{rank}

\DeclareMathOperator{\dir}{dir}
\DeclareMathOperator{\Span}{span}
\DeclareMathOperator*{\argmax}{argmax}

\newcommand{\CP}[1]{\cite[#1]{QL_orderpolytope}}
\newcommand{\CPsecprelim}{Section~2}
\newcommand{\CPdimofFMP}{Lemma~2.5}
\newcommand{\CPpersistence}{Lemma~2.6}
\newcommand{\CPcayley}{Proposition~3.1}
\newcommand{\CPbijection}{Corollary~3.2}
\newcommand{\CPvertical}{Theorem~3.6}
\newcommand{\CPcrossproduct}{Lemma~4.1}
\newcommand{\CPdelta}{Lemma~4.2}
\newcommand{\CPposetheight}{Theorem~4.6}
\newcommand{\CPexp}{Corollary~4.21}
\newcommand{\CPremdelta}{Remark~4.7}
\newcommand{\CPgrowth}{Corollary~4.20}
\newcommand{\CPexample}{Example~3.9}
\newcommand{\CPdeletezeros}{Lemma~4.8}
\newcommand{\CPsupportdim}{Lemmas~2.4 and~2.5}
\newcommand{\CPsharptwo}{Theorem~4.17}
\newcommand{\CPlowerfour}{Proposition~4.25}
\newcommand{\CPopen}{Remark~4.29}
\newcommand{\CPcoeffH}{$H_\ell(d)$}

\newcommand{\jon}[1]{}
\newcommand{\yushan}[1]{}

\begin{document}

\title[Facet normals of Minkowski sums of staircase simplices]{Facet normals of Minkowski sums\hb{} of staircase simplices}

\author[1]{\fnm{Yushan} \sur{Qu}}\email{yushanqu@umich.edu}
\author*[1]{\fnm{Jon} \sur{Lee}}\email{jonxlee@umich.edu}

\affil*[1]{\orgname{University of Michigan}, \orgaddress{\city{Ann Arbor}, \state{MI}, \country{USA}}}

\abstract{\mathversion{normal}Staircase simplices are the simplices of the standard (Freudenthal--Kuhn) triangulation of $[0,1]^d$ and of Stanley's triangulations of order polytopes. We study the facet normals of their Minkowski sums and how the number of summands constrains the size of primitive integer normals. We characterize the normals attainable with a prescribed number of summands in terms of $d-1$ linearly independent zero-one vectors orthogonal to the normal, whose supports have an intersection graph that can be colored with one color per summand. This gives a combinatorial characterization of the least number of summands needed to realize a normal. For two summands, all primitive normals have entries in $\{-1,0,1\}$, and we determine which of them occur. For three summands we bound every entry in absolute value by $3^{(d-2)/3}$ and construct normals with entries of order $\theta^d$, $\theta$ the plastic number; for four summands a second construction gives entries of order $\psi^d$, $\psi$ the supergolden ratio. Since $\psi>3^{1/3}$, the largest coefficients for three and four summands are exponentially separated. For vectors with at most two distinct nonzero entry values, we determine the least number of summands exactly, using a matching-rank formula proved by exterior algebra and representations of symmetric groups. With no restriction on the number of summands, the maximum coefficient equals the classical maximum for facets of $0/1$-polytopes in the same dimension.}

\keywords{Minkowski sum, staircase simplex, facet normal, alcoved polytope, graph coloring, order polytope}

\pacs[MSC Classification]{52B12, 52B05, 05C15, 05E10, 90C11}

\maketitle

\section{Introduction}\label{sec:intro}

The normal fan of a Minkowski sum is the common refinement of the normal fans of its summands~\cite[Proposition~7.12]{ziegler}, so a sum can have facet directions that none of its summands has. This paper asks which directions arise, and how large their coefficients can be, when the summands are the simplices of the standard triangulation of the cube.

Let $d\ge2$. For a permutation $\sigma=(\sigma_1,\ldots,\sigma_d)$ of $[d]$, the \emph{staircase simplex} is
\[
\Delta_\sigma:=\{x\in[0,1]^d:x_{\sigma_1}\le x_{\sigma_2}\le\cdots\le x_{\sigma_d}\}.
\]
The $d!$ staircase simplices triangulate the cube $[0,1]^d$; this is the standard triangulation of the cube, also known as the Freudenthal--Kuhn triangulation~\cite{Kuhn1960}. More generally, for a poset $P$ on $[d]$, the simplices $\Delta_\sigma$ indexed by the linear extensions of $P$ form Stanley's triangulation of its order polytope~\cite{Stanley1986}. We study the facet normals of Minkowski sums
\[
Q=\Delta_{\sigma^0}+\Delta_{\sigma^1}+\cdots+\Delta_{\sigma^n},
\]
where $n\ge0$ and the permutations $\sigma^0,\ldots,\sigma^n$ need not be distinct. Our aim is to characterize the attainable normal directions and determine how the number of summands constrains their coefficients.

Each staircase simplex is alcoved in the sense of Lam and Postnikov~\cite{LamPostnikov2007}: its facet normals are positive multiples of vectors of the form $\mathbf{e}_i$, $-\mathbf{e}_i$, or $\mathbf{e}_i-\mathbf{e}_j$. Minkowski sums of alcoved polytopes need not remain alcoved, so even sums of staircase simplices can have facet-normal directions outside the alcoved class. By contrast, sums of coordinate simplices $\conv\{\mathbf{e}_i:i\in I\}$ are generalized permutahedra~\cite{Postnikov2009}, whose normal fans coarsen the braid fan, so they never leave a fixed finite set of directions. \Cref{fig:sum} shows the smallest example: the sum of two staircase simplices in $\mathbb{R}^3$ with a facet whose normal is not alcoved.

\begin{figure}[t]
\centering
\begin{tikzpicture}[x=1cm,y=1cm]
\filldraw[fill=facelight,draw=black,line width=0.5pt] (-0.578,1.938) -- (0.000,0.000) -- (0.752,2.249) -- cycle;
\filldraw[fill=facelight,draw=black,line width=0.5pt] (-0.578,0.715) -- (0.000,0.000) -- (-0.578,1.938) -- cycle;
\draw[dashed,gray,line width=0.4pt] (-0.578,0.715) -- (0.752,2.249);
\node at (0.09,-0.45) {$\Delta_{(1,3,2)}$};
\node at (1.35,1.12) {$+$};
\filldraw[fill=facelight,draw=black,line width=0.5pt] (3.463,2.249) -- (2.712,0.000) -- (4.041,1.533) -- cycle;
\filldraw[fill=facelight,draw=black,line width=0.5pt] (4.041,1.533) -- (2.712,0.000) -- (4.041,0.311) -- cycle;
\draw[dashed,gray,line width=0.4pt] (4.041,0.311) -- (3.463,2.249);
\node at (3.38,-0.45) {$\Delta_{(2,3,1)}$};
\node at (4.64,1.12) {$=$};
\filldraw[fill=facelight,draw=black,line width=0.5pt] (6.515,3.286) -- (5.763,1.038) -- (7.093,2.571) -- cycle;
\filldraw[fill=facelight,draw=black,line width=0.5pt] (5.763,-0.185) -- (6.341,-0.900) -- (5.763,1.038) -- cycle;
\filldraw[fill=facelight,draw=black,line width=0.5pt] (7.093,2.571) -- (7.671,0.633) -- (8.423,2.882) -- cycle;
\filldraw[fill=facelight,draw=black,line width=0.5pt] (7.671,0.633) -- (6.341,-0.900) -- (7.671,-0.589) -- cycle;
\filldraw[fill=facetshade,draw=black,line width=0.5pt] (5.763,1.038) -- (6.341,-0.900) -- (7.671,0.633) -- (7.093,2.571) -- cycle;
\filldraw[fill=facelight,draw=black,line width=0.5pt] (6.515,3.286) -- (7.093,2.571) -- (8.423,2.882) -- (7.845,3.597) -- cycle;
\filldraw[fill=facelight,draw=black,line width=0.5pt] (7.671,0.633) -- (7.671,-0.589) -- (8.423,1.660) -- (8.423,2.882) -- cycle;
\draw[dashed,gray,line width=0.4pt] (6.515,2.064) -- (6.515,3.286);
\draw[dashed,gray,line width=0.4pt] (6.515,2.064) -- (7.845,3.597);
\draw[dashed,gray,line width=0.4pt] (5.763,-0.185) -- (7.093,0.126);
\draw[dashed,gray,line width=0.4pt] (7.093,0.126) -- (8.423,1.660);
\draw[dashed,gray,line width=0.4pt] (7.671,-0.589) -- (7.093,0.126);
\draw[dashed,gray,line width=0.4pt] (5.763,-0.185) -- (6.515,2.064);
\draw[dashed,gray,line width=0.4pt] (8.423,1.660) -- (7.845,3.597);
\draw[dashed,gray,line width=0.4pt] (7.093,0.126) -- (6.515,2.064);
\node at (7.09,-1.35) {$\Delta_{(1,3,2)}+\Delta_{(2,3,1)}$};
\end{tikzpicture}
\caption{The staircase simplices $\Delta_{(1,3,2)}=\{0\leq x_1\leq x_3\leq x_2\leq1\}$ and $\Delta_{(2,3,1)}=\{0\leq x_2\leq x_3\leq x_1\leq1\}$ and their Minkowski sum, which has $14$ facets. The shaded facet has outward normal $(-1,-1,1)$; the facet with normal $(1,1,-1)$ is on the far side. Neither is a facet normal of any single staircase simplex (\cref{ex:facetnormalconstruction}(ii)). Hidden edges are dashed.}
\label{fig:sum}
\end{figure}
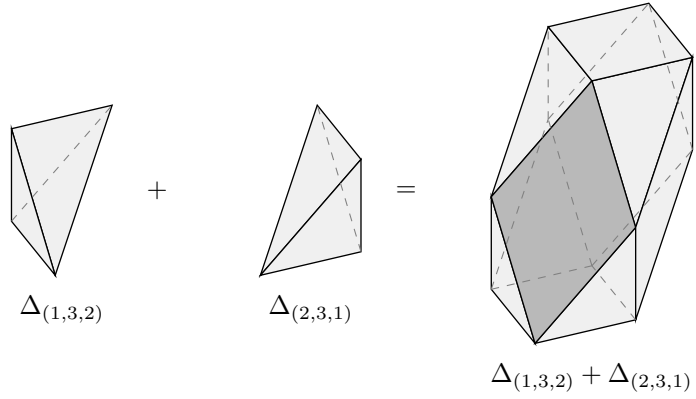

Early, K\"uhne, and Monin~\cite{early2025when,EKM2025fpsac} characterize when Minkowski sums of alcoved polytopes remain alcoved, proving that the sum of a collection is alcoved if and only if every pairwise sum is, and giving a graphical criterion for alcoved simplices. They also identify the common refinement of the normal fans of all alcoved polytopes with the fan of the all-subset arrangement \cite[Proposition~2.11]{early2025when}, and describe the rays of pointed intersections of two root cones by alternating cycles \cite[Proposition~3.6]{early2025when}. Our focus is the dependence on the number of staircase summands.

Every sum $Q$ is an integral polytope, so each of its facet-normal rays has a unique primitive integer representative. Throughout, bounds on coefficients refer to this representative, whose entries have greatest common divisor one. For a nonzero vector $\alpha\in\mathbb{R}^d$, let $\kappa(\alpha)$ be the least number of staircase simplices whose Minkowski sum has $\alpha$ as a facet normal, with $\kappa(\alpha)=\infty$ if no such sum exists. 

This is the second of two papers. The first \cite{QL_orderpolytope} studies, for a poset $P$ on $[d]$, the Cayley embedding $\D$ of the simplices of Stanley's triangulation of the order polytope of $P$, motivated by disjunctive programming \cite{balas2018disjunctive} and by formulations of piecewise-linear functions. It shows that the facets of $\D$ other than those described by the bounds on the indicator variables correspond to the facet-normal rays of the Minkowski sum $M(P)$ of the simplices (the Cayley trick, see \cite{HuberRambauSantos2000}; \CP{\CPcayley, \CPbijection}), it characterizes the posets for which $M(P)$ is alcoved \CP{\CPvertical}, and it bounds the entries of the primitive facet normals of $M(P)$ in terms of the distance of $P$ from a chain \CP{\CPposetheight}. The present paper drops the poset and asks what the number of summands alone implies. We follow the notation of \cite{QL_orderpolytope} and refer to it for the results we need; \cref{sec:orderpolytopes} explains how the results below apply to the sums $M(P)$.

Our first result is a combinatorial characterization of the attainable normals. The starting point, taken from \cite{QL_orderpolytope}, is that the direction space of the face of $\Delta_\sigma$ on which $\alpha$ is maximized is spanned by the indicator vectors of its maximizing blocks, which are blocks of consecutive entries of $\sigma$ with zero $\alpha$-sum (\cref{lem:dimofFMP}); this turns a question about exposed faces into one about zero-sum subsets of $[d]$. We prove that a nonzero vector $\alpha$ is a facet normal of a sum of $n+1$ staircase simplices if and only if there are $d-1$ nonempty subsets of $[d]$ with zero $\alpha$-sum, linearly independent indicator vectors, and an intersection graph of chromatic number at most $n+1$, where two subsets are adjacent when they intersect (\cref{thm:facetnormalcharacterization}). Consequently, $\kappa(\alpha)$ is the least chromatic number over all such families, $\kappa(\alpha)\leq d-1$ whenever it is finite, and $\kappa(\alpha)<\infty$ exactly when the zero-sum indicator vectors span $\alpha^\perp$ (\cref{cor:anysum}), that is, exactly when $\alpha$ spans a ray of the chamber fan of the resonance arrangement (\cref{rem:universal}); \cref{thm:facetnormalcharacterization} refines this unrestricted description by characterizing realizability with a prescribed number of summands.

Our second group of results concerns the growth of primitive facet coefficients. Let $\mathfrak{h}_m(d)$ be the largest absolute entry of a primitive facet normal of a Minkowski sum of $m$ staircase simplices, not necessarily distinct, in $\mathbb{R}^d$; it is nondecreasing in $m$, because adding a summand preserves facet normals. With $\mathrm{Pad}_i$ the Padovan numbers and $\mathcal{N}_i$ Narayana's cows numbers, we prove
\begin{gather*}
\mathfrak{h}_2(d)=1,\qquad \mathrm{Pad}_{d-1}\leq \mathfrak{h}_3(d)\leq3^{(d-2)/3}\quad(d\geq4),\\
\mathcal{N}_d\leq \mathfrak{h}_4(d)\leq12^{(d-1)/4}\quad(d\geq3),
\end{gather*}
and $\mathfrak{h}_m(d)\leq m^{(d-1)/2}$ in general (\cref{thm:inftynormofalpha,thm:padovan,thm:narayana}). For three summands and $d\le6$ we determine $\mathfrak{h}_3(d)$ exactly (\cref{prop:H3small}). The two sequences grow like $\theta^d$ and $\psi^d$ for the plastic number $\theta=1.3247\ldots$ and the supergolden ratio $\psi=1.4655\ldots$ (\cref{cor:H3,thm:narayana}). Since $\psi>3^{1/3}$,
\[
\limsup_{d\to\infty}\mathfrak{h}_3(d)^{1/d}
\le3^{1/3}<\psi
\le\liminf_{d\to\infty}\mathfrak{h}_4(d)^{1/d},
\]
so $\mathfrak{h}_4(d)/\mathfrak{h}_3(d)$ grows at least exponentially with $d$; the four-summand construction requires exactly four summands in every dimension $d\ge5$, and already $\mathfrak{h}_4(5)=3>2=\mathfrak{h}_3(5)$ (\cref{prop:fourinfive}).

When the number of summands is unrestricted, the coefficient problem has a classical interpretation: for $m\geq d-1$, $\mathfrak{h}_m(d)$ equals the largest coefficient in a facet-describing inequality, with coprime integer coefficients, of a full-dimensional $0/1$-polytope in $\mathbb{R}^d$ (\cref{cor:anysum,prop:coeff}). By the bounds of Alon and V\~u~\cite{AlonVu97}, as presented by Ziegler~\cite[Section~5]{Ziegler01}, and the zero-one version of Hadamard's maximal determinant problem, this quantity grows like $d^{d/2}$ up to a factor $2^{O(d)}$.

Our third group of results determines the exact summand count for normals with at most two distinct nonzero entry values. For $\alpha\in\{-1,0,1\}^d$ we give a closed formula for $\kappa(\alpha)$ (\cref{thm:kappa}), which also identifies exactly which vectors in $\{-1,0,1\}^d$ are facet normals of sums of two staircase simplices (\cref{cor:two}). The main case is $\alpha=(a^p,-b^q)$, where $a,b$ are coprime positive integers and $p,q\ge1$ are their multiplicities. Every minimal nonempty subset with zero $\alpha$-sum contains $b$ positive and $a$ negative coordinates. We prove an exact formula for the maximum rank of the indicators in a union of a prescribed number of full matchings, that is, largest sets of pairwise disjoint minimal zero-sum subsets (\cref{thm:matchingrank}). Its proof uses exterior algebra and representations of symmetric groups, and it determines the least number of summands, including when it is finite (\cref{thm:twomagnitude}), and the maximum exposed-face dimension (\cref{thm:exposedface}). Zero padding then covers every vector with at most two distinct nonzero values (\cref{cor:twovalues}).

\smallskip\noindent\emph{Outline.}
\Cref{sec:prelim} gives preliminaries, \cref{sec:characterization} the characterization, \cref{sec:size,sec:three} the coefficient bounds and constructions, \cref{sec:kappa,sec:twomagnitude} the exact summand counts, and \cref{sec:orderpolytopes} the application to order polytopes.

\smallskip\noindent\emph{Notation.}
We use the notation of \cite{QL_orderpolytope}. For a positive integer $d$, let $[d]:=\{1,2,\ldots,d\}$ and $[d_0]:=\{0,1,\ldots,d\}$. For a nonnegative integer $n$, let $N_0:=\{0,1,\ldots,n\}$. We write $\mathbf{e}_i$ for the $i$-th standard unit vector, $\mathbf{e}$ for the all-ones vector, $0$ for the zero vector, and $\mathbf{e}_B:=\sum_{i\in B}\mathbf{e}_i$ for the indicator vector of a set $B$; the dimension will be clear from the context. For $\alpha\in\mathbb{R}^d$, let $\alpha^\perp:=\{x\in\mathbb{R}^d:\alpha^\top x=0\}$. We write $\|\alpha\|_\infty:=\max_i|\alpha_i|$. The symbol $\subset$ denotes inclusion, not necessarily proper.

\section{Preliminaries}\label{sec:prelim}
This section fixes notation and recalls, from \CP{\CPsecprelim}, the facts about staircase simplices and their Minkowski sums that we use. The first paper states them for the linear extensions of a poset, but, as noted there, their proofs use only that the permutations involved are permutations of $[d]$, so they hold for every Minkowski sum of staircase simplices; we restate them in that generality and refer to \cite{QL_orderpolytope} for the proofs.

\subsection{Staircase simplices}\label{subsec:order}
Let $\sigma$ be a permutation of $[d]$. The vertices of $\Delta_\sigma$ are $v^\sigma_r:=\sum_{t=r+1}^d\mathbf{e}_{\sigma_t}$\,, $r\in[d_0]$; thus $v^\sigma_0=\mathbf{e}$ and $v^\sigma_d=0$\,. Equivalently, $\Delta_\sigma=\{x\in\mathbb{R}^d: -x_{\sigma_1}\leq0,\ x_{\sigma_d}\leq1,\ x_{\sigma_r}-x_{\sigma_{r+1}}\leq0 \text{ for } r\in[d-1]\}$, and each of these $d+1$ inequalities describes a facet; the gap coordinates $y_0=x_{\sigma_1}$\,, $y_r=x_{\sigma_{r+1}}-x_{\sigma_r}$ ($r\in[d-1]$), $y_d=1-x_{\sigma_d}$ map $\Delta_\sigma$ affinely onto the standard simplex, with $v^\sigma_r$ corresponding to $y_r=1$ (see \CP{Section~2.1}).

Throughout, $n\geq0$ is an integer and $\sigma^0,\ldots,\sigma^n$ is a sequence of permutations of $[d]$, not necessarily distinct. As in \cite{QL_orderpolytope}, we write $\Delta_k:=\Delta_{\sigma^k}$ and $v^k_r:=v^{\sigma^k}_r$ for $k\in N_0$ and $r\in[d_0]$, and we put $Q:=\sum_{k\in N_0}\Delta_k$\,. When $\sigma^0,\ldots,\sigma^n$ are the linear extensions of a poset $P$, this is the polytope $M(P)$ of \cite{QL_orderpolytope}.

\subsection{Support functions and maximizing blocks}\label{subsec:support}
We use standard terminology for polytopes (see \cite{ziegler}), in the form of \CP{Section~2.4}: an inequality $\alpha^\top x\leq\beta$ that is valid for a polytope $K$ \emph{describes} the face $K\cap\{x:\alpha^\top x=\beta\}$; $h_K(\alpha):=\max_{x\in K}\alpha^\top x$ is the support function, and $F_K(\alpha):=\{x\in K:\alpha^\top x=h_K(\alpha)\}$ is the face it exposes, so $F_K(0)=K$; for full-dimensional $K$, a vector $\alpha\neq0$ is a \emph{facet normal} of $K$ if $\dim F_K(\alpha)=d-1$, and its \emph{facet-normal ray} is $\{\lambda\alpha:\lambda>0\}$; and $\dir C:=\Span\{x-y:x,y\in C\}$ is the direction space of a nonempty convex set $C$, so that $\dim C=\dim\dir C$ and $\dir(A+B)=\dir A+\dir B$. For $k\in N_0$ let $h_k(\alpha):=h_{\Delta_k}(\alpha)$. Support functions and faces are additive under Minkowski sums (see \cite{Schneider2014}): for every $\alpha\in\mathbb{R}^d$,
\begin{equation}\label{eq:minkowskiadd}
\textstyle h_{Q}(\alpha)=\sum_{k\in N_0}h_k(\alpha)\,,\qquad F_{Q}(\alpha)=\sum_{k\in N_0}F_{\Delta_k}(\alpha)\,,
\end{equation}
and $Q$ is full dimensional, because each $\Delta_k$ is.

\smallskip
\noindent\emph{Maximizing blocks.}
Following \CP{Section~2.4}, for $k \in N_0$ and $\alpha \in \mathbb{R}^d$ let $s^k_r(\alpha) := \alpha^\top v^{k}_r = \sum_{t=r+1}^d \alpha_{\sigma^k_t}$ for $r\in[d_0]$, the suffix sums of $\alpha$ along $\sigma^k$, and let $R_k(\alpha) := \argmax_{r \in [d_0]}s^k_r(\alpha)$, with elements $r^{k}_0 < r^{k}_1 < \cdots < r^{k}_{m_k}$\,, where $m_k = m_k(\alpha) := |R_k(\alpha)| - 1$. For $j\in[m_k]$, the $j$-th \emph{maximizing block of $\sigma^k$ with respect to $\alpha$} is the set $B_{k,j}(\alpha) := \{\sigma^k_{r^k_{j-1}+1}\,, \ldots, \sigma^k_{r^k_{j}} \}$ of entries of $\sigma^k$ in positions $r^k_{j-1}+1$ through $r^k_j$\,, so that $v^k_{r^k_{j-1}}-v^k_{r^k_j}=\mathbf{e}_{B_{k,j}(\alpha)}$\,. If $|R_k(\alpha)|=1$, there are no blocks, and if $\alpha=0$, every entry is a block by itself. The following lemma combines \CP{\CPsupportdim}.

The following lemma combines \cite[Lemma~2.4]{QL_orderpolytope}, which is part (a), and \cite[Lemma~2.5]{QL_orderpolytope}, which gives parts (b) and (c). In \cite{QL_orderpolytope} these results are stated for $M(P)$, the sum over the distinct linear extensions of a poset, whereas $Q$ is the sum over an arbitrary sequence of permutations, possibly with repetitions. This makes no difference to the proofs: part (a) concerns a single simplex $\Delta_k$, and, as noted after \cite[Lemma~2.5]{QL_orderpolytope}, the proof of (b) and (c) uses only that each $\sigma^k$ is a permutation of $[d]$. We therefore state the lemma for $Q$ and refer to \cite{QL_orderpolytope} for the proof.

\begin{lemma}[{\CP{\CPsupportdim}}]\label{lem:dimofFMP}
For every $k\in N_0$ and $\alpha\in\mathbb{R}^d$:
\begin{enumerate}[label=(\alph*)]
    \item $h_k(\alpha)=\max_{r\in[d_0]}s_r^k(\alpha)$, and $F_{\Delta_k}(\alpha)=\conv\{v_r^k:r\in R_k(\alpha)\}$;
    \item the vectors $\mathbf{e}_{B_{k,j}(\alpha)}$\,, $1\leq j\leq m_k(\alpha)$, form a basis of $\dir F_{\Delta_k}(\alpha)$, and each of them is orthogonal to $\alpha$; in particular, the blocks $B_{k,1}(\alpha),\ldots,B_{k,m_k}(\alpha)$ are nonempty and pairwise disjoint;
    \item $\dim F_{Q}(\alpha) = \rank \{ \mathbf{e}_{B_{k,j}(\alpha)}: k \in N_0\,,\ 1 \leq j \leq m_k(\alpha) \}$.
\end{enumerate}
In particular, for $\alpha\neq0$ this rank is at most $d-1$, and $\alpha$ is a facet normal of $Q$ if and only if the rank equals $d-1$, that is, the vectors $\mathbf{e}_{B_{k,j}(\alpha)}$ span $\alpha^\perp$.
\end{lemma}

Thus $h_k(\alpha)$ is the largest suffix sum of $\alpha$ along $\sigma^k$, including the empty suffix, and the maximizing blocks of $\sigma^k$ are the runs of $\sigma^k$ between consecutive positions at which this maximum is attained.

Facet normals persist when summands are added; this is the common-refinement property of normal fans (see \cite[Proposition~7.12]{ziegler}).

\begin{lemma}[{\CP{\CPpersistence}}]\label{lem:persistenceaddition}
Let $K \subset \mathbb{R}^d$ be a full-dimensional polytope, and let
$R \subset\mathbb{R}^d$ be a nonempty polytope. Every facet-normal ray of $K$ is a facet-normal ray of $K+R$.
\end{lemma}

The converse fails: the common refinement may have new rays, which are the subject of this paper.

\smallskip
\noindent\emph{Intersection graphs.}
For nonempty subsets $B_1\,,\ldots,B_{d-1}\subset[d]$, let $\Gamma(B_1\,,\ldots,B_{d-1})$ be their \emph{intersection graph}: its vertex set is $[d-1]$, and $j\neq j'$ are adjacent when $B_j\cap B_{j'}\neq\varnothing$. Partitioning $[d-1]$ into classes $J_0\,,\ldots,J_n$ such that $B_j\cap B_{j'}=\varnothing$ for all $j\neq j'$ in the same class is the same as properly coloring $\Gamma(B_1\,,\ldots,B_{d-1})$ with $n+1$ colors, the classes being the color classes. We write $\chi(\Gamma)$ for the chromatic number of a graph $\Gamma$.

\section{Characterization}\label{sec:characterization}
The vertices $v^k_r$\,, the sets $R_k(\alpha)$, and the maximizing blocks $B_{k,j}(\alpha)$ are those of \cref{sec:prelim}, for the permutations $\sigma^0,\ldots,\sigma^n$ and the sum $Q=\sum_{k\in N_0}\Delta_{\sigma^k}$\,.

\begin{theorem}\label{thm:facetnormalcharacterization}
Let $d\geq2$ and $n\geq0$ be integers, and let $\alpha\in\mathbb{R}^d\setminus\{0\}$. The following statements are equivalent.
\begin{enumerate}[label=(\arabic*),ref=(\arabic*)]
\item\label{thm:fnc:1} There exist permutations $\sigma^0,\ldots,\sigma^n$ of $[d]$, not necessarily distinct, such that $\alpha$ is a facet normal of $Q:=\sum_{k\in N_0}\Delta_{\sigma^k}$\,.
\item\label{thm:fnc:2} There exist nonempty subsets $B_1\,,\ldots,B_{d-1}\subset[d]$ satisfying the following conditions:
\begin{enumerate}[label=(\alph*),ref=(\alph*)]
\item\label{thm:fnc:a} $\alpha^\top\mathbf{e}_{B_j}=\sum_{i\in B_j}\alpha_i=0$ for every $j\in[d-1]$;
\item\label{thm:fnc:b} the indicator vectors $\mathbf{e}_{B_j}$\,, $j\in[d-1]$, are linearly independent;
\item\label{thm:fnc:c} $B_1\,,\ldots,B_{d-1}$ can be partitioned into at most $n+1$ classes such that any two subsets in the same class are disjoint; equivalently, $\chi(\Gamma(B_1\,,\ldots,B_{d-1}))\leq n+1$.
\end{enumerate}
\end{enumerate}
Moreover, if \ref{thm:fnc:2} holds with classes $J_0\,,\ldots,J_n\subset[d-1]$ (some possibly empty), then the permutations in \ref{thm:fnc:1} can be chosen so that $\mathbf{e}_{B_j}\in\dir F_{\Delta_{\sigma^k}}(\alpha)$ for every $k\in N_0$ and every $j\in J_k$\,.
\end{theorem}

\begin{proof}
    \emph{\ref{thm:fnc:1} implies \ref{thm:fnc:2}.} Suppose that $\alpha$ is a facet normal of $Q$. By \cref{lem:dimofFMP},
    \[
    \dir F_Q(\alpha)=\Span\{\mathbf{e}_{B_{k,j}(\alpha)}:k\in N_0\,,\ 1\leq j\leq m_k(\alpha)\},
    \]
    and this space has dimension $d-1$. Choose $d-1$ linearly independent vectors in this family, and let $B_1\,,\ldots,B_{d-1}$ be the corresponding maximizing blocks. Each $B_j$ is nonempty and satisfies $\alpha^\top\mathbf{e}_{B_j}=0$ by \cref{lem:dimofFMP}, so \ref{thm:fnc:a} and \ref{thm:fnc:b} hold. For \ref{thm:fnc:c}, let $J_k$ be the set of $j\in[d-1]$ such that $B_j$ is a maximizing block of $\sigma^k$ (if a block occurs for several $k$, assign $j$ to one of them). Two subsets $B_j\neq B_{j'}$ with $j,j'\in J_k$ are distinct maximizing blocks of the same permutation $\sigma^k$, hence disjoint by \cref{lem:dimofFMP}; and there are at most $n+1$ nonempty classes.

    \smallskip
    \emph{\ref{thm:fnc:2} implies \ref{thm:fnc:1}: the permutations.} Let $B_1\,,\ldots,B_{d-1}$ satisfy \ref{thm:fnc:a}--\ref{thm:fnc:c}, with classes $J_0\,,\ldots,J_n$\,, some possibly empty. Fix $k\in N_0$ and let $T_k:=[d]\setminus\bigcup_{j\in J_k}B_j$\,. Because the subsets $B_j$\,, $j\in J_k$\,, are pairwise disjoint, they partition $[d]\setminus T_k$\,. Let $\sigma^k$ be a permutation of $[d]$ that lists
    \begin{enumerate}[label=(\roman*)]
    \item first, the elements $i\in T_k$ with $\alpha_i<0$;
    \item next, the subsets $B_j$\,, $j\in J_k$\,, one after another, listing within each $B_j$ first the elements $i$ with $\alpha_i>0$, then those with $\alpha_i=0$, then those with $\alpha_i<0$;
    \item finally, the elements $i\in T_k$ with $\alpha_i\geq0$.
    \end{enumerate}
    The order within each of the three parts, within each sign class, and among the subsets in (ii) is arbitrary. Thus each $B_j$\,, $j\in J_k$\,, occupies a set of consecutive positions $a_j+1,\ldots,b_j$ of $\sigma^k$, that is, $B_j=\{\sigma^k_{a_j+1},\ldots,\sigma^k_{b_j}\}$, and these position intervals are pairwise disjoint. If $J_k=\varnothing$, then $\sigma^k$ lists the $i$ with $\alpha_i<0$ before the $i$ with $\alpha_i\geq0$.

    \smallskip
    \emph{An upper bound on $\Delta_{\sigma^k}$\,.} Recall that $s^k_r(\alpha) = \alpha^\top v^k_r=\sum_{t=r+1}^{d}\alpha_{\sigma^k_t}$ for $r\in[d_0]$, and put $\beta_k:=\sum_{i\in T_k,\,\alpha_i\geq0}\alpha_i$\,; we show that $\beta_k=h_k(\alpha)$. First, $\alpha^\top v^k_r\leq \beta_k$ for all $r\in[d_0]$. Split the sum $\sum_{t=r+1}^{d}\alpha_{\sigma^k_t}$ according to the parts (i)--(iii) of $\sigma^k$. The terms from part (i) are negative, so their sum is at most $0$; the terms from part (iii) are nonnegative, so their sum is at most $\beta_k$\,. For a subset $B_j$\,, $j\in J_k$\,, in positions $a_j+1,\ldots,b_j$\,, the terms from $B_j$ form a suffix of its listing: all of $B_j$\,, with sum $0$ by \ref{thm:fnc:a}, or nothing, or a proper nonempty suffix. In the last case, if the suffix contains a positive entry, then the complementary prefix consists of positive entries, because these come first within $B_j$\,, so by \ref{thm:fnc:a} the suffix sum is the negative of a positive number; otherwise the suffix has no positive entry. So the terms from each $B_j$ sum to at most $0$, and $\alpha^\top v^k_r\leq \beta_k$\,. Equality holds when $r$ is the number of entries in part (i), because the suffix then consists of the subsets $B_j$\,, $j\in J_k$\,, each with zero $\alpha$-sum, and of part (iii); hence $\beta_k=h_k(\alpha)$ by \cref{lem:dimofFMP}(a).

    \smallskip
    \emph{Attaining the bound.} Fix $j\in J_k$\,. For $r\in\{a_j,b_j\}$, the sum $\sum_{t=r+1}^{d}\alpha_{\sigma^k_t}$ contains no term from part (i), all terms from part (iii), and, from each $B_{j'}$ with $j'\in J_k$\,, either all terms $\alpha_{\sigma^k_t}$, $a_{j'}<t\leq b_{j'}$\,, or none of them, because the position intervals are pairwise disjoint. The terms from each $B_{j'}$ sum to $0$ by \ref{thm:fnc:a}, so $\alpha^\top v^k_{a_j}=\alpha^\top v^k_{b_j}=\beta_k$\,. By \cref{lem:dimofFMP}, $v^k_{a_j},v^k_{b_j}\in F_{\Delta_{\sigma^k}}(\alpha)$, and therefore $v^k_{a_j}-v^k_{b_j}=\sum_{t=a_j+1}^{b_j}\mathbf{e}_{\sigma^k_t}=\mathbf{e}_{B_j}\in\dir F_{\Delta_{\sigma^k}}(\alpha)$,
    which is the last assertion of the theorem.

    \smallskip
    \emph{Conclusion.} By \eqref{eq:minkowskiadd} and the additivity of direction spaces (\cref{subsec:support}), $\dir F_Q(\alpha)=\sum_{k\in N_0}\dir F_{\Delta_{\sigma^k}}(\alpha)$.
    Every $j\in[d-1]$ lies in some $J_k$\,, so $\mathbf{e}_{B_1}\,,\ldots,\mathbf{e}_{B_{d-1}}\in\dir F_Q(\alpha)$, and these vectors are linearly independent by \ref{thm:fnc:b}. Hence $\dim F_Q(\alpha)\geq d-1$. On the other hand, $F_Q(\alpha)$ lies in the hyperplane $\{x\in\mathbb{R}^d:\alpha^\top x=h_Q(\alpha)\}$, so $\dim F_Q(\alpha)\leq d-1$. As $Q$ is full dimensional (\cref{subsec:support}), $\alpha$ is a facet normal of $Q$. This proves \ref{thm:fnc:1}.
\end{proof}

We comment on the statement of \cref{thm:facetnormalcharacterization}. Because $\dim\alpha^\perp=d-1$, conditions \ref{thm:fnc:a} and \ref{thm:fnc:b} together say that $\mathbf{e}_{B_1}\,,\ldots,\mathbf{e}_{B_{d-1}}$ form a basis of $\alpha^\perp$. The classes in \ref{thm:fnc:c} are the color classes of a proper coloring of $\Gamma(B_1\,,\ldots,B_{d-1})$ with $n+1$ colors, which is the form in which \ref{thm:fnc:c} is used in \cref{cor:anysum}. We use three consequences throughout. First, conditions \ref{thm:fnc:a}--\ref{thm:fnc:c} are unchanged when $\alpha$ is multiplied by a nonzero scalar, possibly negative, and no integrality of $\alpha$ is needed; so the same holds for the least number of summands defined below, and from \cref{lem:crossproduct} on, whenever we bound the entries of $\alpha$, we normalize it to be a primitive integer vector. Second, classes may be empty, so if \ref{thm:fnc:2} holds for $n$, it holds for every larger $n$; correspondingly, by \cref{lem:persistenceaddition}, a facet normal of a sum of $n+1$ staircase simplices is one of every sum with further summands. Third, because $\Delta_\sigma+\Delta_\sigma=2\Delta_\sigma$ has the same normal fan as $\Delta_\sigma$\,, and the normal fan of a Minkowski sum is the common refinement of the normal fans of its summands (see \cite[Proposition~7.12]{ziegler}), repeated summands may be discarded: statement \ref{thm:fnc:1} is equivalent to \emph{$\alpha$ is a facet normal of the sum of at most $n+1$ distinct staircase simplices.}

\begin{example}\label{ex:facetnormalconstruction}
    Let $d=3$.
    \begin{enumerate}[label=(\roman*)]
        \item For $\alpha=(1,-1,0)$, the disjoint subsets $\{1,2\}$ and $\{3\}$ satisfy \ref{thm:fnc:a}--\ref{thm:fnc:c} with one class; the construction gives $\sigma^0=(1,2,3)$, and indeed $x_1-x_2\leq0$ describes a facet of $\Delta_{(1,2,3)}$\,.
        \item For $\alpha=(1,1,-1)$, the subsets with $\alpha$-sum zero are $\{1,3\}$ and $\{2,3\}$, with linearly independent indicator vectors. They intersect, so $\chi(\Gamma)=2$ and two summands are needed. With $B_1=\{1,3\}$ in class $0$ and $B_2=\{2,3\}$ in class $1$, the construction gives
        \[
        \sigma^0=(1,3,2),\qquad \sigma^1=(2,3,1).
        \]
        For $\sigma^0$ the values $\alpha^\top v^0_r$\,, $r=0,1,2,3$, are $1,0,1,0$, so $h_0(\alpha)=1$ is attained at $v^0_0$ and $v^0_2$\,, with $v^0_0-v^0_2=\mathbf{e}_{\{1,3\}}$; symmetrically, $\mathbf{e}_{\{2,3\}}\in\dir F_{\Delta_{\sigma^1}}(\alpha)$. Hence $\alpha$ is a facet normal of $\Delta_{\sigma^0}+\Delta_{\sigma^1}$\,. Because $\alpha$ is not a positive multiple of any $\pm\mathbf{e}_i$ or $\mathbf{e}_i-\mathbf{e}_j$\,, it is not a facet normal of any single staircase simplex. Both $\sigma^0$ and $\sigma^1$ are linear extensions of the antichain on $[3]$, illustrating the sufficient condition in \cref{sec:orderpolytopes}; for that poset, the same vector and the same two blocks appear in \CP{\CPexample}, where they give a vertical facet of the Cayley embedding that is not described by lifting a facet-describing inequality of any single simplex.
        \item For $\alpha=(1,1,-2)$, the only nonempty zero-sum subset is $\{1,2,3\}$, so \ref{thm:fnc:b} fails, and $\alpha$ is a facet normal of no sum of staircase simplices.
    \end{enumerate}
\end{example}

The entries of a facet normal are bounded by minors of the indicator matrix of the subsets in \cref{thm:facetnormalcharacterization}; we recall the statement from \CP{\CPcrossproduct}.

\begin{lemma}[{Facet normals as generalized cross products, \CP{\CPcrossproduct}}]\label{lem:crossproduct}
    Let $B_1\,,\ldots,B_{d-1}\subset[d]$ be nonempty subsets whose indicator vectors are linearly independent. Let $A$ be the $(d-1)\times d$ matrix with rows $\mathbf{e}_{B_j}^\top$\,, let $\widehat A^i$ be obtained from $A$ by deleting column $i$, and let $\gamma_i:=(-1)^{i-1}\det\widehat A^i$ for $i\in[d]$ (the determinant of the empty matrix being one); $\gamma(B_1\,,\ldots,B_{d-1}):=\gamma$ is the \emph{generalized cross product} of $\mathbf{e}_{B_1}\,,\ldots,\mathbf{e}_{B_{d-1}}$\,. Then $\gamma\in\mathbb{Z}^d\setminus\{0\}$ and $\ker A=\Span\{\gamma\}$. Consequently, if $\alpha\in\mathbb{R}^d\setminus\{0\}$ satisfies $\sum_{i\in B_j}\alpha_i=0$ for all $j\in[d-1]$, then $\alpha=t\gamma$ for some $t\neq0$; if moreover $\alpha$ is a primitive integer vector, then $\alpha=\pm\gamma/g$ with $g:=\gcd(\gamma_1\,,\ldots,\gamma_d)$, and in particular $|\alpha_i|\leq|\det\widehat A^i|$ for every $i\in[d]$.
\end{lemma}

For a nonzero vector $\alpha\in\mathbb{R}^d$, let $\kappa(\alpha)$ be the least number of staircase simplices, repetitions allowed, whose Minkowski sum has facet normal $\alpha$, and set $\kappa(\alpha):=\infty$ if no such collection exists.

\begin{corollary}\label{cor:anysum}
    Let $d\geq2$ and $\alpha\in\mathbb{R}^d\setminus\{0\}$. Then $\kappa(\alpha)<\infty$, that is, $\alpha$ is a facet normal of some Minkowski sum of staircase simplices, if and only if the vectors $\mathbf{e}_B$ with $\varnothing\neq B\subset[d]$ and $\alpha^\top\mathbf{e}_B=0$ span $\alpha^\perp$. If so, then
    \[
    \kappa(\alpha)=\min\Bigl\{\chi\bigl(\Gamma(B_1\,,\ldots,B_{d-1})\bigr):\ B_1\,,\ldots,B_{d-1}\subset[d]\text{ nonempty, satisfying \ref{thm:fnc:a} and \ref{thm:fnc:b}}\Bigr\},
    \]
    where \ref{thm:fnc:a} and \ref{thm:fnc:b} are the conditions of \cref{thm:facetnormalcharacterization}. Moreover, in this case $\kappa(\alpha)\leq d-1$, and this bound is attained, for instance by $\alpha=(1,-1,\ldots,-1)$.
\end{corollary}

\begin{proof}
    A nonzero $0/1$ vector is $\mathbf{e}_B$ for a nonempty $B\subset[d]$, and it lies in $\alpha^\perp$ exactly when $\alpha^\top\mathbf{e}_B=\sum_{i\in B}\alpha_i=0$. If \ref{thm:fnc:1} of \cref{thm:facetnormalcharacterization} holds for some $n$, then so does \ref{thm:fnc:2}, and its conditions \ref{thm:fnc:a} and \ref{thm:fnc:b} give $d-1$ linearly independent such vectors $\mathbf{e}_{B_1}\,,\ldots,\mathbf{e}_{B_{d-1}}$ in the $(d-1)$-dimensional space $\alpha^\perp$, which therefore span it. Conversely, if these vectors span $\alpha^\perp$, choose $d-1$ linearly independent ones $\mathbf{e}_{B_1}\,,\ldots,\mathbf{e}_{B_{d-1}}$\,; they satisfy \ref{thm:fnc:a} and \ref{thm:fnc:b}, and \ref{thm:fnc:c} holds with the $d-1$ classes $\{1\},\ldots,\{d-1\}$, that is, with $n+1=d-1$.

    For each $n\geq0$, statement \ref{thm:fnc:2} of \cref{thm:facetnormalcharacterization} holds if and only if some nonempty $B_1\,,\ldots,B_{d-1}\subset[d]$ satisfying \ref{thm:fnc:a} and \ref{thm:fnc:b} have $\chi(\Gamma(B_1\,,\ldots,B_{d-1}))\leq n+1$, because the classes in \ref{thm:fnc:c} are the color classes of a proper coloring of $\Gamma(B_1\,,\ldots,B_{d-1})$ with $n+1$ colors; so the formula for $\kappa(\alpha)$ is the equivalence of \ref{thm:fnc:1} and \ref{thm:fnc:2} applied to each $n$. The bound $\kappa(\alpha)\leq d-1$ is the case of the $d-1$ classes $\{1\},\ldots,\{d-1\}$ above. For $\alpha=(1,-1,\ldots,-1)$, a subset $B\subset[d]$ has $\alpha^\top\mathbf{e}_B=|B\cap\{1\}|-|B\setminus\{1\}|$, so the nonempty subsets with $\alpha^\top\mathbf{e}_B=0$ are exactly $\{1,j\}$, $2\leq j\leq d$. Their indicator vectors $\mathbf{e}_1+\mathbf{e}_j$ are linearly independent (the $j$-th coordinate of $\mathbf{e}_1+\mathbf{e}_j$ is $1$, and that of $\mathbf{e}_1+\mathbf{e}_{j'}$ is $0$ for $j'\neq j$), so every family satisfying \ref{thm:fnc:a} and \ref{thm:fnc:b} consists of all $d-1$ of them; any two of them intersect in $\{1\}$, so $\Gamma(\{1,2\},\ldots,\{1,d\})$ is the complete graph on $[d-1]$ and its chromatic number is $d-1$.
\end{proof}

The criterion in \cref{cor:anysum} is finite: it is a rank computation over at most $2^d-1$ vectors. The value $\kappa((1,-1,\ldots,-1))=d-1$ is also a special case of \cref{thm:kappa} in \cref{sec:kappa}, with $p_{\min}=1$ and $p_{\max}=d-1$.

Removing zero coordinates does not change the least number of summands. For the sum $M(P)$ over the linear extensions of a fixed poset, \CP{\CPdeletezeros} gives the analogous reduction, under a hypothesis on singleton maximizing blocks that is not needed here, because the summands are unrestricted.

\begin{lemma}[Zero padding]\label{lem:zeropadding}
    Let $d\geq2$, let $\alpha\in\mathbb{R}^d\setminus\{0\}$, and let $S\subset[d]$ with $|S|\geq2$ and $\alpha_i=0$ for all $i\in[d]\setminus S$. Let $\alpha_S:=(\alpha_i)_{i\in S}\in\mathbb{R}^S$ be the restriction of $\alpha$ to $S$, where $\mathbb{R}^S$ is identified with $\mathbb{R}^{|S|}$ by listing $S$ in increasing order. Then $\kappa(\alpha)=\kappa(\alpha_S)$.
\end{lemma}

\begin{proof}
    Put $s:=|S|$ and $Z:=[d]\setminus S$; then $\alpha_S\neq0$, and $\alpha^\top\mathbf{e}_B=\alpha_S^\top\mathbf{e}_{B\cap S}$ for every $B\subset[d]$, where $\mathbf{e}_{B\cap S}\in\mathbb{R}^S$. We show that, for each $n$, statement \ref{thm:fnc:2} of \cref{thm:facetnormalcharacterization} holds for $\alpha$ if and only if it holds for $\alpha_S$\,; this gives $\kappa(\alpha)=\kappa(\alpha_S)$.

    If $B_1\,,\ldots,B_{s-1}\subset S$ satisfy \ref{thm:fnc:a}--\ref{thm:fnc:c} for $\alpha_S$ with classes $J_0\,,\ldots,J_n$\,, add the singletons $\{i\}$, $i\in Z$, to $J_0$\,. The resulting $d-1$ subsets have zero $\alpha$-sum, the new ones are disjoint from each other and from every $B_j$\,, and their indicator vectors are linearly independent, because a coordinate $i\in Z$ occurs only in $\mathbf{e}_i$\,.

    Conversely, if $B'_1\,,\ldots,B'_{d-1}$ satisfy \ref{thm:fnc:a}--\ref{thm:fnc:c} for $\alpha$, their indicator vectors span $\alpha^\perp$. The coordinate projection $\pi:\mathbb{R}^d\to\mathbb{R}^S$ is surjective and $\alpha^\top x=\alpha_S^\top\pi(x)$, so $\pi(\alpha^\perp)=\alpha_S^\perp$, and the vectors $\pi(\mathbf{e}_{B'_j})=\mathbf{e}_{B'_j\cap S}$ span the $(s-1)$-dimensional space $\alpha_S^\perp$. Choose $s-1$ linearly independent ones; the corresponding sets $B'_j\cap S$ are nonempty, have zero $\alpha_S$-sum, and, with the induced classes, inherit the disjointness in \ref{thm:fnc:c}.
\end{proof}

\begin{remark}\label{rem:alphafromblocks}
    Since $\alpha$ may be replaced by $-\alpha$ in \cref{thm:facetnormalcharacterization}, $-\alpha$ is a facet normal of some sum of $n+1$ staircase simplices whenever $\alpha$ is, though not necessarily of the same one. For instance, $\mathbf{e}_i$ is a facet normal of $\Delta_\sigma$ exactly when $\sigma_d=i$, and $-\mathbf{e}_i$ exactly when $\sigma_1=i$, so no single staircase simplex has both, whereas $\Delta_\sigma+\Delta_\tau$ with $\sigma_d=\tau_1=i$ has both.
\end{remark}

\begin{remark}\label{rem:universal}
    All attainable facet-normal rays are realized simultaneously by one polytope, the sum $U_d:=\sum_\sigma\Delta_\sigma$ of all $d!$ staircase simplices. Indeed, if $\kappa(\alpha)<\infty$, then $\alpha$ is a facet normal of a sum of at most $d-1$ \emph{distinct} staircase simplices (\cref{cor:anysum} and the comment after \cref{thm:facetnormalcharacterization}), and adding the remaining staircase simplices preserves its facet-normal ray (\cref{lem:persistenceaddition}); conversely, every facet normal of $U_d$ is attainable. We note, without using it, that $U_d=\sum_{\varnothing\neq S\subset[d]}(|S|-1)!\,(d-|S|)!\,[0,\mathbf{e}_S]$ is a zonotope (a geometric form of the generalized Hunt--Dyson formula, see \cite{Gioev2002}). Its normal fan is the chamber fan of the resonance (all-subset) arrangement $\mathcal{A}_d=\{\mathbf{e}_S^\perp:\varnothing\neq S\subset[d]\}$ \cite{GutekunstMeszarosPetersen2021,Kuhne2023}; see also \cite[Proposition~2.11]{early2025when}. Thus $\kappa(\alpha)<\infty$ exactly when $\alpha$ spans a ray of the chamber fan of $\mathcal{A}_d$ (\cref{cor:anysum}), and \cref{thm:facetnormalcharacterization} determines how many summands each such ray needs.
\end{remark}

\section{Size of facet normals}\label{sec:size}

As in \CP{Section~4.1}, for a nonnegative integer $j$ let $\delta_j$ be the greatest absolute value of the determinant of a $j\times j$ zero-one matrix, with $\delta_0:=1$; this is the zero-one version of Hadamard's maximal determinant problem, with values $1,1,1,2,3,5,9,32,56,144$ for $j=0,\ldots,9$ (see \cite{OEISA003432}). Bordering a matrix by a corner entry $1$ shows that $\delta_j$ is nondecreasing in $j$. We also use the classical Hadamard-type bound for zero-one matrices.

\begin{lemma}[{\cite[Lemma~24]{Ziegler01}; see also \CP{\CPdelta}}]\label{lem:deltahadamard}
    For every integer $j\geq0$, $\delta_j\leq (j+1)^{(j+1)/2}/2^j$.
\end{lemma}

The following bound is also recorded in \CP{\CPremdelta}.

\begin{corollary}\label{cor:deltabound}
    Every primitive facet normal $\alpha$ of every Minkowski sum of staircase simplices in $\mathbb{R}^d$ satisfies $\lVert\alpha\rVert_\infty\leq\delta_{d-1}$\,.
\end{corollary}

\begin{proof}
    Take subsets as in \cref{cor:anysum}, and let $A$ be their $(d-1)\times d$ indicator matrix; by \cref{lem:crossproduct}, $|\alpha_i|\leq|\det\widehat A^i|\leq\delta_{d-1}$\,.
\end{proof}

For integers $m\geq1$ and $d\geq2$, let $\mathfrak{h}_m(d)$ be the maximum of $\lVert\alpha\rVert_\infty$ over all primitive facet normals $\alpha$ of Minkowski sums of $m$ staircase simplices, not necessarily distinct, in $\mathbb{R}^d$; in the notation of \cref{thm:inftynormofalpha}, $m=n+1$. By the comment after \cref{thm:facetnormalcharacterization}, $\mathfrak{h}_m(d)$ is nondecreasing in $m$, and $\alpha$ is a facet normal of a sum of $m$ staircase simplices for every $m\geq\kappa(\alpha)$. Moreover, by \cref{cor:anysum}, $\mathfrak{h}_m(d)$ is constant for $m\geq d-1$; let $\mathfrak{h}(d):=\max_{m\geq1}\mathfrak{h}_m(d)=\mathfrak{h}_{d-1}(d)$. \Cref{cor:deltabound} says $\mathfrak{h}(d)\leq\delta_{d-1}$\,. (The quantity \CPcoeffH{} of \CP{Section~4.2} is the analogous maximum over the posets $P$ on $[d]$ with $\ell(P)=\ell$, where $\ell(P)$ is the least number of elements whose deletion leaves a chain; it depends on the order relations of $P$ rather than on the number of summands, and we use a different letter to keep the two apart. See \cref{sec:orderpolytopes}.)

The quantity $\mathfrak{h}(d)$ is classical. For a full-dimensional $0/1$-polytope in $\mathbb{R}^d$, write each facet-describing inequality with integer coefficients whose greatest common divisor is $1$, and let $\operatorname{coeff}(d)$ be the largest absolute value of a coefficient of $x_1\,,\ldots,x_d$ in such an inequality, over all full-dimensional $0/1$-polytopes in $\mathbb{R}^d$ (see \cite[Section~5]{Ziegler01}).

\begin{proposition}\label{prop:coeff}
    For every $d\geq2$, $\mathfrak{h}(d)=\operatorname{coeff}(d)$.
\end{proposition}

\begin{proof}
    Let $\alpha$ be a primitive facet normal of a Minkowski sum of staircase simplices. By \cref{cor:anysum}, there are nonempty $B_1\,,\ldots,B_{d-1}\subset[d]$ with $\alpha^\top\mathbf{e}_{B_j}=0$ for $j\in[d-1]$ and $\mathbf{e}_{B_1}\,,\ldots,\mathbf{e}_{B_{d-1}}$ linearly independent, so $0,\mathbf{e}_{B_1}\,,\ldots,\mathbf{e}_{B_{d-1}}$ are $d$ affinely independent $0/1$ points in the hyperplane $\alpha^\perp$, and their affine hull is $\alpha^\perp$. Choose $i\in[d]$ with $\alpha_i\neq0$; then $\mathbf{e}_i\notin\alpha^\perp$, so the $0/1$-simplex $\conv\{0,\mathbf{e}_{B_1}\,,\ldots,\mathbf{e}_{B_{d-1}}\,,\mathbf{e}_i\}$ is full dimensional, and its facet $\conv\{0,\mathbf{e}_{B_1}\,,\ldots,\mathbf{e}_{B_{d-1}}\}\subset\alpha^\perp$ is described by $\alpha^\top x\leq0$ if $\alpha_i<0$ and by $-\alpha^\top x\leq0$ if $\alpha_i>0$. Because $\alpha$ is primitive, the coefficients of this inequality are coprime integers, so $\lVert\alpha\rVert_\infty\leq\operatorname{coeff}(d)$.

    Conversely, let $c^\top x\leq c_0$ be a facet-describing inequality of a full-dimensional $0/1$-polytope $\Po\subset\mathbb{R}^d$, with $c\in\mathbb{Z}^d$, $c_0\in\mathbb{Z}$, and $\gcd(c_1\,,\ldots,c_d\,,c_0)=1$. The facet $\Po\cap\{x:c^\top x=c_0\}$ contains $d$ affinely independent vertices $p^0,\ldots,p^{d-1}\in\{0,1\}^d$ of $\Po$. Let $\varphi:\mathbb{R}^d\to\mathbb{R}^d$ be the affine map that replaces $x_j$ by $1-x_j$ for every $j$ with $p^0_j=1$ and leaves the other coordinates unchanged, and let $c'\in\mathbb{Z}^d$ be given by $c'_j:=-c_j$ if $p^0_j=1$ and $c'_j:=c_j$ otherwise. Then $\varphi$ maps $\{0,1\}^d$ onto itself, $\varphi(p^0)=0$, and $(c')^\top\varphi(x)=c^\top x-c^\top p^0=c^\top x-c_0$ for every $x\in\mathbb{R}^d$. Hence $\varphi(p^1),\ldots,\varphi(p^{d-1})$ are $0/1$ vectors in $(c')^\perp$, linearly independent because $\varphi(p^0)=0,\varphi(p^1),\ldots,\varphi(p^{d-1})$ are affinely independent; so they span the $(d-1)$-dimensional space $(c')^\perp$. Because $c_0=c^\top p^0$ is an integer combination of $c_1\,,\ldots,c_d$\,, we have $\gcd(c_1\,,\ldots,c_d)=\gcd(c_1\,,\ldots,c_d\,,c_0)=1$, so $c'$ is primitive. By \cref{cor:anysum}, $c'$ is a facet normal of a Minkowski sum of staircase simplices, and $\lVert c'\rVert_\infty=\max_j|c_j|$. Hence $\mathfrak{h}(d)\geq\operatorname{coeff}(d)$.
\end{proof}

\begin{remark}\label{rem:coeffvalues}
    The values $\operatorname{coeff}(d)=1,2,3,5,9,18$ for $d=3,\ldots,8$ were obtained by complete enumeration, as reported in \cite[Section~5]{Ziegler01}, which also gives the bounds $\operatorname{coeff}(9)\geq42$ and $\operatorname{coeff}(10)\geq96$. Because $\delta_{d-1}=1,2,3,5,9,32$ for $d=3,\ldots,8$, the bound of \cref{cor:deltabound} is attained for $2\leq d\leq7$ (for $d=2$, $\mathfrak{h}(2)=\delta_1=1$) but not for $d=8$. Asymptotically, Alon and V\~u~\cite{AlonVu97} showed that
    \[
    \frac{(d-1)^{(d-1)/2}}{2^{2d+o(d)}}\leq\operatorname{coeff}(d)\leq\delta_{d-1}\leq\frac{d^{d/2}}{2^{d-1}}
    \]
    (see \cite[Corollary~26]{Ziegler01}; the last inequality is \cref{lem:deltahadamard}). By \cref{prop:coeff}, the same holds for $\mathfrak{h}(d)$, so \cref{cor:deltabound} has the correct order $d^{d/2}$ up to a factor $2^{O(d)}$. \Cref{thm:inftynormofalpha} improves on it for few summands.
\end{remark}

\begin{lemma}\label{lem:productbound}
    Let $M$ be a square zero-one matrix, and write $\operatorname{supp}(M_{\ell,\cdot})$ for the set of column indices of the entries $1$ in row $\ell$. Suppose that the row indices are partitioned into sets $\mathcal R$, $\mathcal T_1$\,, and $\mathcal T_2$ such that $\operatorname{supp}(M_{\ell,\cdot})\cap\operatorname{supp}(M_{\ell',\cdot})=\varnothing$ for all $\ell\neq\ell'$ in $\mathcal T_1$ and for all $\ell\neq\ell'$ in $\mathcal T_2$\,. Then $|\det M|\leq\nu$, where $\nu$ is the number of injective maps $\varphi$ from $\mathcal R$ to the column indices with $\varphi(\ell)\in\operatorname{supp}(M_{\ell,\cdot})$ for every $\ell\in\mathcal R$ (so $\nu=1$ if $\mathcal R=\varnothing$). Moreover, $\nu\leq\prod_{\ell\in\mathcal R}|\operatorname{supp}(M_{\ell,\cdot})|$, with equality if the supports of the rows in $\mathcal R$ are pairwise disjoint.
\end{lemma}

\begin{proof}
    Put $\mathcal T:=\mathcal T_1\cup\mathcal T_2$ and negate the rows in $\mathcal T_2$\,, which changes $\det M$ only by a sign. Then every column of $M[\mathcal T,\cdot]$ has at most one entry $1$ and at most one entry $-1$, so $M[\mathcal T,\cdot]$ is totally unimodular, being a submatrix of the incidence matrix of a directed graph (see \cite{NWbook}). The generalized Laplace expansion along the rows in $\mathcal R$ gives
    \[
    \det M=\sum_{S}\pm\det M[\mathcal R,S]\,\det M[\mathcal T,\overline S],
    \]
    where $S$ runs over the sets of $|\mathcal R|$ column indices and $\overline S$ is the complement of $S$. Each factor $\det M[\mathcal T,\overline S]$ lies in $\{-1,0,1\}$, and $|\det M[\mathcal R,S]|$ is at most the permanent of $M[\mathcal R,S]$, which counts the bijections $\varphi:\mathcal R\to S$ with $\varphi(\ell)\in\operatorname{supp}(M_{\ell,\cdot})$ for all $\ell\in\mathcal R$. Each injective map counted by $\nu$ arises exactly once, for $S=\varphi(\mathcal R)$, so $|\det M|\leq\nu$. Dropping injectivity gives the product bound, and if the supports in $\mathcal R$ are pairwise disjoint, every such map is injective.
\end{proof}

For integers $s\geq0$ and $m\geq1$, let $\Pi(s,m)$ be the largest product of $m$ nonnegative integers with sum at most $s$, and put $\Pi(s,0):=1$. The maximum is attained when every part equals $\lfloor s/m\rfloor$ or $\lceil s/m\rceil$, because replacing two parts that differ by at least $2$ by two parts closer together does not decrease the product. By the inequality of arithmetic and geometric means, $\Pi(s,m)\leq(s/m)^m$; and because $x\mapsto x\log(s/x)$ is increasing on $(0,s/e]$ and $s/3<s/e$, we get $\Pi(s,m)\leq3^{s/3}$ whenever $m\leq s/3$. Relatedly, for an integer $r\geq0$, let
\[
\mathfrak{B}(r):=\max\Bigl\{\textstyle\prod_{j=1}^{\ell}a_j:\ \ell\geq0,\ a_j\in\mathbb{Z}_{\geq1}\,,\ \sum_{j=1}^{\ell}a_j\leq r\Bigr\},
\]
the empty product being $1$. Omitting parts equal to $1$, replacing a part $a\geq5$ by the two parts $3$ and $a-3$, replacing three parts equal to $2$ by two parts equal to $3$, and splitting a part $4$ into $2,2$ shows that
\begin{equation}\label{eq:Bformula}
\mathfrak{B}(0)=\mathfrak{B}(1)=1,\qquad
\mathfrak{B}(r)=\begin{cases}
3^{r/3}, & r\geq2,\ r\equiv0\pmod3,\\
4\cdot3^{(r-4)/3}, & r\geq2,\ r\equiv1\pmod3,\\
2\cdot3^{(r-2)/3}, & r\geq2,\ r\equiv2\pmod3.
\end{cases}
\end{equation}
In particular $\mathfrak{B}(r)\leq3^{r/3}$, and $\mathfrak{B}(r)\leq \mathfrak{B}(r-1)+\mathfrak{B}(r-3)$ for every $r\geq5$, with equality for $r\equiv0,2\pmod3$, while for $r=3k+1\geq7$ it reads $12\cdot3^{k-2}\leq13\cdot3^{k-2}$. For $r\geq2$, $\mathfrak{B}(r)$ is the function $B(r)$ of \CP{\CPexp}, where the value $\mathfrak{B}(d-2)$ is attained by a primitive facet normal of $M(P)$ for a chain with $\lfloor(d-1)/3\rfloor+1$ isolated elements adjoined; we write $\mathfrak{B}$ to avoid a clash with the subsets $B_j$\,. In \cref{thm:inftynormofalpha}\ref{thm:size:b}, the same number bounds the entries for three summands.

\begin{lemma}[An incomplete color class]\label{lem:incompleteclass}
    Let $M$ be a nonsingular $s\times s$ zero-one matrix, $s\geq1$, and suppose that its row indices $[s]$ are partitioned into three classes such that $\operatorname{supp}(M_{\ell,\cdot})\cap\operatorname{supp}(M_{\ell',\cdot})=\varnothing$ for all $\ell\neq\ell'$ in the same class. Then $|\det M|\leq \mathfrak{B}(s-1)$.
\end{lemma}

\begin{proof}
    For a class $J\subset[s]$ put $C(J):=\bigcup_{\ell\in J}\operatorname{supp}(M_{\ell,\cdot})\subset[s]$, the set of columns covered by the rows in $J$. If $C(J)=C(J')=[s]$ for two distinct classes $J$ and $J'$, then, the supports within each class being pairwise disjoint, $\sum_{\ell\in J}M_{\ell,\cdot}=\mathbf{e}^\top=\sum_{\ell\in J'}M_{\ell,\cdot}$\,, a nontrivial linear dependence among the rows of $M$ (as $J\cap J'=\varnothing$), contrary to nonsingularity. So some class $J$, possibly empty, has $|C(J)|\leq s-1$; its row supports are pairwise disjoint, so $\sum_{\ell\in J}|\operatorname{supp}(M_{\ell,\cdot})|=|C(J)|\leq s-1$.
    By \cref{lem:productbound} with $\mathcal R:=J$ and the two other classes as $\mathcal T_1$ and $\mathcal T_2$\,, $|\det M|\leq\prod_{\ell\in J}|\operatorname{supp}(M_{\ell,\cdot})|$, a product of $|J|$ positive integers with sum $|C(J)|\leq s-1$; by the definition of $\mathfrak{B}$, it is at most $\mathfrak{B}(s-1)$ (and it is $1$ if $J=\varnothing$).
\end{proof}

\begin{theorem}[Size of facet normals]\label{thm:inftynormofalpha}
    Let $d\geq2$ and $n\geq0$ be integers, let $Q\subset\mathbb{R}^d$ be the Minkowski sum of $n+1$ staircase simplices, not necessarily distinct, and let $\alpha\in\mathbb{Z}^d$ be a primitive facet normal of $Q$.
    \begin{enumerate}[label=(\alph*),ref=(\alph*)]
        \item\label{thm:size:a} For at most two summands ($n\leq1$), $\alpha\in\{-1,0,1\}^d$.
        \item\label{thm:size:b} For at most three summands ($n\leq2$), $\lVert\alpha\rVert_\infty\leq \mathfrak{B}(d-2)\leq 3^{(d-2)/3}$.
        \item\label{thm:size:c} For at most four summands ($n\leq3$), $\lVert\alpha\rVert_\infty\leq 12^{(d-1)/4}$.
        \item\label{thm:size:d} In general, $\lVert\alpha\rVert_\infty\leq\min\{(n+1)^{(d-1)/2},\ \delta_{d-1}\}$.
    \end{enumerate}
\end{theorem}

\begin{proof}
    By \cref{thm:facetnormalcharacterization}, there exist nonempty subsets $B_1\,,\ldots,B_{d-1}\subset[d]$ satisfying \ref{thm:fnc:a}--\ref{thm:fnc:c}: $\alpha^\top\mathbf{e}_{B_j}=0$ for $j\in[d-1]$, the vectors $\mathbf{e}_{B_1}\,,\ldots,\mathbf{e}_{B_{d-1}}$ are linearly independent, and $[d-1]$ is partitioned into at most $n+1$ classes such that $B_j\cap B_{j'}=\varnothing$ for $j\neq j'$ in the same class. Adding empty classes, we may use $n'+1$ classes $J_0\,,\ldots,J_{n'}$ for any $n'\geq n$. Let $A$, $\widehat A^i$ and $\gamma$ be as in \cref{lem:crossproduct}. By that lemma, $\alpha=\pm\gamma/g$ with $g:=\gcd(\gamma_1\,,\ldots,\gamma_d)$, and $|\alpha_i|\leq|\det\widehat A^i|$ for $i\in[d]$.
    Fix $i\in[d]$.

    \smallskip
    \emph{\ref{thm:size:a}.} Use two classes $J_0\,,J_1$\,. The row $j$ of $\widehat A^i$ has support $B_j\setminus\{i\}$, so \cref{lem:productbound} with $\mathcal R=\varnothing$, $\mathcal T_1=J_0$\,, and $\mathcal T_2=J_1$ gives $\det\widehat A^i\in\{-1,0,1\}$ for every $i$, so $\gamma\in\{-1,0,1\}^d$. A nonzero vector with entries in $\{-1,0,1\}$ is primitive, so $g=1$ and $\alpha=\pm\gamma\in\{-1,0,1\}^d$.

    \smallskip
    \emph{\ref{thm:size:b}.} Use three classes $J_0\,,J_1\,,J_2$\,. The matrix $\widehat A^i$ is a $(d-1)\times(d-1)$ zero-one matrix whose row $j$ has support $B_j\setminus\{i\}$; for $j\neq j'$ in the same class these supports are disjoint, so $J_0\,,J_1\,,J_2$ partition the rows of $\widehat A^i$ as required in \cref{lem:incompleteclass}. If $\widehat A^i$ is singular, then $\det\widehat A^i=0$; otherwise \cref{lem:incompleteclass} with $s=d-1$ gives $|\det\widehat A^i|\leq \mathfrak{B}(d-2)$.

    \smallskip
    \emph{\ref{thm:size:c}.} If $\det\widehat A^{\,i}=0$, the desired bound is immediate; hence assume $\det\widehat A^{\,i}\ne0$. Use four classes $J_0\,,J_1\,,J_2\,,J_3$\,, labeled so that $m_0\leq m_1\leq m_2\leq m_3$\,, where $m_k:=|J_k|$, and put $s:=d-1$; note that $m_0+m_1+m_2+m_3=s$. By \cref{lem:productbound} with $\mathcal R=J_0\cup J_1$\,, $\mathcal T_1=J_2$\,, and $\mathcal T_2=J_3$\,, $|\det\widehat A^i|$ is at most the number of injective maps $\varphi:J_0\cup J_1\to[d]\setminus\{i\}$ with $\varphi(j)\in B_j$ for every $j\in J_0\cup J_1$\,. The sets $B_j\setminus\{i\}$, $j\in J_0$\,, are pairwise disjoint subsets of $[d]\setminus\{i\}$, so $\sum_{j\in J_0}|B_j\setminus\{i\}|\leq s$, and the restriction of $\varphi$ to $J_0$ can be chosen in at most $\prod_{j\in J_0}|B_j\setminus\{i\}|\leq\Pi(s,m_0)\leq(s/m_0)^{m_0}$ ways. Once the restriction of $\varphi$ to $J_0$ is fixed, each value $\varphi(j)$, $j\in J_1$\,, lies in $B_j\setminus(\{i\}\cup\varphi(J_0))$, and these sets are pairwise disjoint subsets of $[d]\setminus(\{i\}\cup\varphi(J_0))$, a set with $s-m_0$ elements; so the restriction of $\varphi$ to $J_1$ can be chosen in at most $\Pi(s-m_0,m_1)\leq((s-m_0)/m_1)^{m_1}$ ways. Because $m_1$ is the smallest of $m_1\,,m_2\,,m_3$\,, whose sum is $s-m_0$\,, we have $m_1\leq(s-m_0)/3$, so $((s-m_0)/m_1)^{m_1}\leq3^{(s-m_0)/3}$, because $x\mapsto x\log\bigl((s-m_0)/x\bigr)$ is increasing on $\bigl(0,(s-m_0)/e\bigr]$, as noted before \cref{lem:incompleteclass} (the convention $(s/0)^0=1$ covers empty classes). Hence $\log|\det\widehat A^i|\leq f(m_0)$, where
    \[
    f(x):=x\log(s/x)+\bigl((s-x)/3\bigr)\log3,\qquad 0\leq x\leq s/4 .
    \]
    For $0<x\leq s/4$ we have $f'(x)=\log(s/x)-1-(\log3)/3\geq\log4-1-(\log3)/3>0$, so $f(m_0)\leq f(s/4)=(s/4)\log12$, that is, $|\det\widehat A^i|\leq12^{(d-1)/4}$.

    \smallskip
    \emph{\ref{thm:size:d}.} Use $n+1$ classes $J_0\,,\ldots,J_n$\,. For each $i'\in[d]\setminus\{i\}$ and each class $J_k$\,, at most one of the pairwise disjoint subsets $B_j$\,, $j\in J_k$\,, contains $i'$; so the column of $\widehat A^i$ indexed by $i'$ has at most $n+1$ ones and Euclidean norm at most $\sqrt{n+1}$. Hadamard's inequality gives $|\det\widehat A^i|\leq (n+1)^{(d-1)/2}$. The bound $\delta_{d-1}$ is \cref{cor:deltabound}.
\end{proof}

In terms of $\mathfrak{h}_m(d)$, \cref{thm:inftynormofalpha} says $\mathfrak{h}_2(d)=1$, $\mathfrak{h}_3(d)\leq\min\{\mathfrak{B}(d-2),\delta_{d-1}\}$ (for example, $\mathfrak{B}(6),\mathfrak{B}(7),\mathfrak{B}(8)=9,12,18$, while $\delta_7,\delta_8,\delta_9=32,56,144$), $\mathfrak{h}_4(d)\leq12^{(d-1)/4}$, and $\mathfrak{h}_m(d)\leq\min\{m^{(d-1)/2},\delta_{d-1}\}$. The bounds in \ref{thm:size:b} and \ref{thm:size:c} should be combined with \ref{thm:size:d}. Because $12^{1/4}\approx1.861<2$, part \ref{thm:size:c} always improves the bound $4^{(d-1)/2}=2^{d-1}$ of \ref{thm:size:d} for four summands; the analogous counting does not improve \ref{thm:size:d} for five summands ($60^{1/5}\approx2.27>\sqrt5$).  Concerning sharpness: \ref{thm:size:a} is attained, for example by $\mathbf{e}_1-\mathbf{e}_2$\,; for many summands, the order $d^{d/2}$ of $\delta_{d-1}$ is correct up to a factor $2^{O(d)}$ (\cref{rem:coeffvalues}); for three and four summands, see \cref{cor:H3,cor:separation}. For three summands and $d\leq6$, \cref{prop:H3small} below determines $\mathfrak{h}_3(d)$ exactly.

\section{Three and four summands}\label{sec:three}

We begin with a small example: $\alpha=(1,1,-1,-2)$ needs three summands. The zero-sum subsets $\{1,3\}$, $\{2,3\}$, and $\{1,2,4\}$ have linearly independent indicator vectors, so $\kappa(\alpha)\leq3$ by \cref{thm:facetnormalcharacterization}, while $\alpha_4=-2$ rules out two summands by \cref{thm:inftynormofalpha}\ref{thm:size:a}. Thus $\kappa(\alpha)=3$, and $\lVert\alpha\rVert_\infty=2=\mathfrak{h}_3(4)$ (\cref{prop:H3small}).

Exponentially large entries already occur for three summands. Define the \emph{Padovan numbers} by $\mathrm{Pad}_0=\mathrm{Pad}_1=\mathrm{Pad}_2:=1$ and $\mathrm{Pad}_i:=\mathrm{Pad}_{i-2}+\mathrm{Pad}_{i-3}$ for $i\geq3$ (see \cite{OEISA000931}); they satisfy the identity
\begin{equation}\label{eq:padovanidentity}
\mathrm{Pad}_i=\mathrm{Pad}_{i-1}+\mathrm{Pad}_{i-5}\qquad(i\geq5),
\end{equation}
because $\mathrm{Pad}_i-\mathrm{Pad}_{i-1}=(\mathrm{Pad}_{i-2}+\mathrm{Pad}_{i-3})-(\mathrm{Pad}_{i-3}+\mathrm{Pad}_{i-4})=\mathrm{Pad}_{i-2}-\mathrm{Pad}_{i-4}=\mathrm{Pad}_{i-5}$\,. In particular, they are nondecreasing: $\mathrm{Pad}_0,\ldots,\mathrm{Pad}_4=1,1,1,2,2$, and $\mathrm{Pad}_i-\mathrm{Pad}_{i-1}=\mathrm{Pad}_{i-5}>0$ for $i\geq5$. Let $\theta=1.3247\ldots$ be the \emph{plastic number}, the real root of $t^3=t+1$; the other two roots of $t^3-t-1$ have modulus $\theta^{-1/2}<1$, so $\mathrm{Pad}_i=\Theta(\theta^{\,i})$. 

\begin{theorem}\label{thm:padovan}
    Let $d\geq4$, and label the coordinates of $\mathbb{R}^d$ as $0,1,\ldots,d-1$. Define $\alpha\in\mathbb{Z}^d$ by $\alpha_i:=(-1)^i\mathrm{Pad}_i$ for $0\leq i\leq d-1$. Then $\alpha$ is a primitive facet normal of a Minkowski sum of three staircase simplices, $\lVert\alpha\rVert_\infty=\mathrm{Pad}_{d-1}$\,, and $\kappa(\alpha)=3$.
\end{theorem}

\begin{proof}
    Consider the subsets
    \begin{gather*}
    B^1:=\{0,1\},\qquad B^2:=\{1,2\},\qquad B^3:=\{0,2,3\},\qquad B^4:=\{3,4\},\\
    B^i:=\{i-5,i-1,i\}\qquad(5\leq i\leq d-1).
    \end{gather*}
    The largest element of $B^i$ is $i$, so exactly $B^1,\ldots,B^{d-1}$ lie in $\{0,\ldots,d-1\}$. We verify conditions \ref{thm:fnc:a}--\ref{thm:fnc:c} of \cref{thm:facetnormalcharacterization} with three classes.

    \emph{Zero sums.} For $B^1,\ldots,B^4$, the sums $\alpha^\top\mathbf{e}_{B^i}$ are $1-1$, $-1+1$, $1+1-2$, and $-2+2$. For $i\geq5$, the indices $i-5$ and $i-1$ have parity opposite to $i$, so $\alpha^\top\mathbf{e}_{B^i}=(-1)^i(\mathrm{Pad}_i-\mathrm{Pad}_{i-1}-\mathrm{Pad}_{i-5})=0$ by \eqref{eq:padovanidentity}.

    \emph{Linear independence.} Delete the coordinate $0$ and order the rows $\mathbf{e}_{B^1}^\top,\ldots,\mathbf{e}_{B^{d-1}}^\top$ by $i$. Since $i$ is the largest element of $B^i$, the resulting $(d-1)\times(d-1)$ matrix, with columns $1,\ldots,d-1$, is lower triangular with unit diagonal; hence the $d-1$ indicator vectors are linearly independent.

    \emph{Three classes.} Put $B^i$ into the class indexed by $i\bmod3$. We claim that two subsets $B^i$ and $B^j$ with $i<j$ and $i\equiv j\pmod3$ are disjoint; note that then $j-i\geq3$. If $i\geq5$, every element of $B^j=\{j-5,j-1,j\}$ that could lie in $B^i=\{i-5,i-1,i\}$ is $j-5$ or $j-1$ (as $j>i$), and the coincidences $j-5\in\{i-1,i\}$ and $j-1=i$ give $j-i\in\{4,5,1\}$, none of which is divisible by $3$. If $i\leq4$, then $j\geq5$, and $B^i\cap B^j\neq\varnothing$ would require $j-5\in B^i$ or $j-1\in B^i$ (as $j>4\geq\max B^i$), that is, $j\in\{5,6\}$ for $i=1$, $j\in\{6,7\}$ for $i=2$, $j\in\{5,7,8\}$ for $i=3$, or $j\in\{5,8,9\}$ for $i=4$; none of these values of $j$ is congruent to $i$ modulo $3$. Finally, $B^1\cap B^4=\varnothing$. Thus the subsets in each class are pairwise disjoint, and \ref{thm:fnc:c} holds with $n+1=3$.

    \emph{Conclusion.} By \cref{thm:facetnormalcharacterization}, $\alpha$ is a facet normal of a Minkowski sum of three staircase simplices. It is primitive because $\alpha_0=\mathrm{Pad}_0=1$, and $\lVert\alpha\rVert_\infty=\mathrm{Pad}_{d-1}$ because the $\mathrm{Pad}_i$ are nondecreasing. Since $\mathrm{Pad}_{d-1}\geq \mathrm{Pad}_3=2$ for $d\geq4$, \cref{thm:inftynormofalpha}\ref{thm:size:a} rules out two summands, so $\kappa(\alpha)=3$.
\end{proof}

\begin{corollary}\label{cor:H3}
    For $d\geq4$, $\mathrm{Pad}_{d-1}\leq \mathfrak{h}_3(d)\leq \mathfrak{B}(d-2)\leq3^{(d-2)/3}$. In particular,
    \[
    \theta\ \leq\ \liminf_{d\to\infty}\mathfrak{h}_3(d)^{1/d}\ \leq\ \limsup_{d\to\infty}\mathfrak{h}_3(d)^{1/d}\ \leq\ 3^{1/3},
    \]
    with $\theta=1.3247\ldots$ and $3^{1/3}=1.4422\ldots$.
\end{corollary}

\begin{proof}
    The lower bound follows from \cref{thm:padovan}, and the upper bound follows from \cref{thm:inftynormofalpha}\ref{thm:size:b}; then $\mathrm{Pad}_{d-1}=\Theta(\theta^{\,d})$ gives the limit bounds.
\end{proof}

A parallel construction separates four summands from three. The role of the Padovan numbers is played by Narayana's cows sequence (see \cite{OEISA000930}): define $\mathcal{N}_1=\mathcal{N}_2=\mathcal{N}_3:=1$ and $\mathcal{N}_i:=\mathcal{N}_{i-1}+\mathcal{N}_{i-3}$ for $i\geq4$, and let $\psi=1.4655\ldots$ be the \emph{supergolden ratio}, the real root of $t^3=t^2+1$; the other two roots of $t^3-t^2-1$ have modulus $\psi^{-1/2}<1$, so $\mathcal{N}_i=\Theta(\psi^{\,i})$. Note that $\psi>3^{1/3}$: the polynomial $t^3-t^2-1$ is increasing for $t\geq1$ and takes the value $2-3^{2/3}<0$ at $t=3^{1/3}$ (because $9>8$). The $\mathcal{N}_i$ are nondecreasing, because $\mathcal{N}_1=\mathcal{N}_2=\mathcal{N}_3$ and $\mathcal{N}_i-\mathcal{N}_{i-1}=\mathcal{N}_{i-3}>0$ for $i\geq4$.

\begin{theorem}\label{thm:narayana}
    Let $d\geq3$. Define $\alpha\in\mathbb{Z}^d$ by $\alpha_i:=(-1)^i\mathcal{N}_i$ for $i\in[d]$. Then $\alpha$ is a primitive facet normal of a Minkowski sum of four staircase simplices, so $\kappa(\alpha)\leq4$; moreover $\lVert\alpha\rVert_\infty=\mathcal{N}_d$ and $\psi^{\,d-3}\leq \mathcal{N}_d\leq\psi^{\,d-1}$.
\end{theorem}

\begin{proof}
    Consider the subsets
    \[
    B^2:=\{1,2\},\qquad B^3:=\{2,3\},\qquad B^i:=\{i-3,i-1,i\}\quad(4\leq i\leq d),
    \]
    of which there are $d-1$. We verify conditions \ref{thm:fnc:a}--\ref{thm:fnc:c} of \cref{thm:facetnormalcharacterization} with four classes.

    \emph{Zero sums.} For $B^2$ and $B^3$, the sums $\alpha^\top\mathbf{e}_{B^i}$ are $-1+1$ and $1-1$. For $i\geq4$, the indices $i-3$ and $i-1$ have parity opposite to $i$, so $\alpha^\top\mathbf{e}_{B^i}=(-1)^i(\mathcal{N}_i-\mathcal{N}_{i-1}-\mathcal{N}_{i-3})=0$ by the recurrence.

    \emph{Linear independence.} Delete the coordinate $1$ and order the rows $\mathbf{e}_{B^2}^\top,\ldots,\mathbf{e}_{B^d}^\top$ by $i$. Since $i$ is the largest element of $B^i$, the resulting $(d-1)\times(d-1)$ matrix, with columns $2,\ldots,d$, is lower triangular with unit diagonal; hence the $d-1$ indicator vectors are linearly independent.

    \emph{Four classes.} Put $B^i$ into the class indexed by $i\bmod4$. Every $B^i$ is contained in the interval $\{i-3,\ldots,i\}$, and if $i<i'$ lie in the same class, then $i'-i\geq4$, so $i<i'-3$ and the intervals $\{i-3,\ldots,i\}$ and $\{i'-3,\ldots,i'\}$ are disjoint. Thus the subsets in each class are pairwise disjoint, and \ref{thm:fnc:c} holds with $n+1=4$.

    \emph{Conclusion.} By \cref{thm:facetnormalcharacterization}, $\alpha$ is a facet normal of a Minkowski sum of four staircase simplices, so $\kappa(\alpha)\leq4$. It is primitive because $\alpha_1=-\mathcal{N}_1=-1$, and $\lVert\alpha\rVert_\infty=\mathcal{N}_d$ because the $\mathcal{N}_i$ are nondecreasing. For the bracketing, multiplying $\psi^3=\psi^2+1$ by $\psi^{i-6}$ gives $\psi^{i-3}=\psi^{i-4}+\psi^{i-6}$, so the sequences $\psi^{i-3}$ and $\psi^{i-1}$ satisfy the same recurrence as $\mathcal{N}_i$\,. They bracket the initial values, $\psi^{-2},\psi^{-1},\psi^{0}\leq1=\mathcal{N}_1,\mathcal{N}_2,\mathcal{N}_3\leq\psi^{0},\psi^{1},\psi^{2}$, and if $\psi^{i-4}\leq \mathcal{N}_{i-1}$ and $\psi^{i-6}\leq \mathcal{N}_{i-3}$, then $\psi^{i-3}=\psi^{i-4}+\psi^{i-6}\leq \mathcal{N}_{i-1}+\mathcal{N}_{i-3}=\mathcal{N}_i$\,; the upper bound follows in the same way. Hence $\psi^{\,d-3}\leq \mathcal{N}_d\leq\psi^{\,d-1}$ by induction on $d$.
\end{proof}

\begin{corollary}[Exponential separation of three and four summands]\label{cor:separation}
    Let $d\geq3$.
    \begin{enumerate}[label=(\alph*),ref=(\alph*)]
        \item\label{cor:sep:a} $\mathfrak{h}_4(d)\geq \mathcal{N}_d\geq\psi^{\,d-3}$.
        \item\label{cor:sep:b} $\mathfrak{h}_4(d)/\mathfrak{h}_3(d)\geq\bigl(3^{2/3}/\psi^3\bigr)\bigl(\psi/3^{1/3}\bigr)^d$, which tends to infinity; consequently
        \[
        \limsup_{d\to\infty}\mathfrak{h}_3(d)^{1/d}\ \leq\ 3^{1/3}=1.4422\ldots\ <\ 1.4655\ldots=\psi\ \leq\ \liminf_{d\to\infty}\mathfrak{h}_4(d)^{1/d}.
        \]
        \item\label{cor:sep:c} The vector $\alpha$ of \cref{thm:narayana} satisfies $\kappa(\alpha)=4$, not merely $\kappa(\alpha)\leq4$, for every $d\geq9$.
    \end{enumerate}
\end{corollary}

\begin{proof}
    \ref{cor:sep:a} is \cref{thm:narayana}. For \ref{cor:sep:b}, \cref{thm:inftynormofalpha}\ref{thm:size:b} gives $\mathfrak{h}_3(d)\leq \mathfrak{B}(d-2)\leq3^{(d-2)/3}$, so
    \[
    \frac{\mathfrak{h}_4(d)}{\mathfrak{h}_3(d)}\ \geq\ \frac{\psi^{\,d-3}}{3^{(d-2)/3}}\ =\ \frac{3^{2/3}}{\psi^3}\Bigl(\frac{\psi}{3^{1/3}}\Bigr)^{d},
    \]
    which tends to infinity because $\psi>3^{1/3}$; taking $d$-th roots gives the limit inequalities.

    For \ref{cor:sep:c}, we claim that $\mathcal{N}_d>\mathfrak{B}(d-2)$ for every $d\geq9$. Then $\mathcal{N}_d>\mathfrak{h}_3(d)$ by \cref{thm:inftynormofalpha}\ref{thm:size:b}, so $\alpha$ is not a facet normal of any Minkowski sum of at most three staircase simplices, and $\kappa(\alpha)=4$ because $\kappa(\alpha)\leq4$ by \cref{thm:narayana}. The three base cases are
    \[
    (\mathcal{N}_9\,,\mathcal{N}_{10}\,,\mathcal{N}_{11})=(13,19,28)\ >\ (12,18,27)=\bigl(\mathfrak{B}(7),\mathfrak{B}(8),\mathfrak{B}(9)\bigr),
    \]
    and for $d\geq12$, using the induction hypotheses $\mathcal{N}_{d-1}>\mathfrak{B}(d-3)$ and $\mathcal{N}_{d-3}>\mathfrak{B}(d-5)$ and the inequality $\mathfrak{B}(r)\leq \mathfrak{B}(r-1)+\mathfrak{B}(r-3)$ for $r=d-2\geq5$, noted after \eqref{eq:Bformula},
    \[
    \mathcal{N}_d=\mathcal{N}_{d-1}+\mathcal{N}_{d-3}\ >\ \mathfrak{B}(d-3)+\mathfrak{B}(d-5)\ \geq\ \mathfrak{B}(d-2),
    \]
    which completes the induction.
\end{proof}

Because $\mathfrak{h}_m(d)$ is nondecreasing in $m$ (\cref{sec:size}), \cref{cor:separation} also gives, for every $m\geq4$,
\[
\liminf_{d\to\infty}\mathfrak{h}_m(d)^{1/d}\ \geq\ \psi:
\]
every summand count from four upwards is exponentially separated from three.

We close this section with exact values in small dimensions, obtained by combining the upper bounds of \cref{sec:size} with the constructions above. For small orders, the bound of \cref{lem:incompleteclass} can be sharpened.

\begin{lemma}\label{lem:smallorders}
    Let $s\in\{4,5\}$, let $M$ be a nonsingular $s\times s$ zero-one matrix, and suppose that $[s]$ is partitioned into three classes such that $\operatorname{supp}(M_{\ell,\cdot})\cap\operatorname{supp}(M_{\ell',\cdot})=\varnothing$ for all $\ell\neq\ell'$ in the same class. Then $|\det M|\leq2$ if $s=4$, and $|\det M|\leq3$ if $s=5$.
\end{lemma}

\begin{proof}
    No row of $M$ is zero, as $M$ is nonsingular. We use two reductions. (A) If some row has exactly one entry $1$, expanding along it gives $|\det M|=|\det M'|$ for an $(s-1)\times(s-1)$ zero-one submatrix $M'$ whose rows inherit the three classes with disjoint supports. (B) If two rows in the same class have supports $\{1,2\}$ and $\{3,4\}$ (after relabeling columns), then subtracting column $1$ from column $2$ and column $3$ from column $4$, which preserves $\det M$, turns these rows into $\mathbf{e}_1^\top$ and $\mathbf{e}_3^\top$; expanding along them gives $|\det M|=|\det N|$, where $N$ consists of the other $s-2$ rows in columns $2$, $4$, and $5,\ldots,s$, with entries in $\{-1,0,1\}$ in columns $2$ and $4$ (differences of zero-one entries) and entries of $M$ elsewhere.

    \emph{The case $s=4$.} By (A), we may assume that every row has at least two entries $1$ (otherwise $|\det M|\leq\delta_3=2$). Some class then contains two rows, whose supports are disjoint subsets of $[4]$ of size at least $2$, hence of size $2$; by (B), $|\det M|=|\det N|$ for a $2\times2$ matrix $N$ with entries in $\{-1,0,1\}$, so $|\det M|\leq2$.

    \emph{The case $s=5$.} By (A) and the case $s=4$, we may assume that every row has at least two entries $1$ (otherwise $|\det M|\leq2$, or $\det M=0$ if $M'$ is singular). Then a class contains at most two rows, so the classes have sizes $2,2,1$. The two classes of size $2$ cannot both cover all five columns, because then both would sum to $\mathbf{e}^\top$ (as in the proof of \cref{lem:incompleteclass}); let $J$ be one that does not. Its two supports are disjoint, of size at least $2$, and contained in at most four columns, so both have size $2$, and by (B), $|\det M|=|\det N|$ for the $3\times3$ matrix
    \[
    N=\begin{pmatrix}a&b&c\\a'&b'&c'\\x&y&z\end{pmatrix}
    \]
    formed by the two rows $\ell_1\,,\ell_2$ of the other class of size $2$ and the row $\ell_3$ of the class of size $1$, in columns $2$, $4$, and $5$. Because the supports of $\ell_1$ and $\ell_2$ are disjoint, each of $(M_{\ell_11},M_{\ell_21})$ and $(M_{\ell_12},M_{\ell_22})$ lies in $\{0,\mathbf{e}_1,\mathbf{e}_2\}$, so $(a,a')\in\{0,\pm\mathbf{e}_1,\pm\mathbf{e}_2,\pm(\mathbf{e}_1-\mathbf{e}_2)\}$; likewise $(b,b')$, while $(c,c')\in\{(0,0),(1,0),(0,1)\}$. A $2\times2$ matrix with columns in $\{0,\pm\mathbf{e}_1,\pm\mathbf{e}_2,\pm(\mathbf{e}_1-\mathbf{e}_2)\}$ is totally unimodular (a submatrix of the incidence matrix of a directed graph, see \cite{NWbook}), so $|ab'-a'b|\leq1$; moreover $x,y\in\{-1,0,1\}$ and $z\in\{0,1\}$. Expanding along the third column,
    \[
    \det N=c\,(a'y-b'x)-c'\,(ay-bx)+z\,(ab'-a'b).
    \]
    If $(c,c')=(0,0)$, then $|\det N|\leq|ab'-a'b|\leq1$; if $(c,c')=(1,0)$, then $|\det N|\leq|a'y|+|b'x|+|ab'-a'b|\leq3$; and $(c,c')=(0,1)$ is symmetric.
\end{proof}

\begin{proposition}\label{prop:H3small}
    $\mathfrak{h}_3(2)=\mathfrak{h}_3(3)=1$, $\mathfrak{h}_3(4)=\mathfrak{h}_3(5)=2$, and $\mathfrak{h}_3(6)=3$.
\end{proposition}

\begin{proof}
    \emph{Upper bounds.} For $d=2,3,4$, \cref{cor:deltabound} gives $\mathfrak{h}_3(d)\leq\delta_{d-1}=1,1,2$. For $d=5,6$, let $\alpha$ be a primitive facet normal of a Minkowski sum of three staircase simplices, and fix $i\in[d]$ with $\alpha_i\neq0$. As in the proof of \cref{thm:inftynormofalpha}\ref{thm:size:b}, $|\alpha_i|\leq|\det\widehat A^i|$, where $\widehat A^i$ is a zero-one matrix of order $d-1$ whose rows are partitioned into three classes with pairwise disjoint supports within each class. If $\widehat A^i$ is singular, then $\det\widehat A^i=0$; otherwise \cref{lem:smallorders} gives $|\det\widehat A^i|\leq2$ for $d=5$ and $|\det\widehat A^i|\leq3$ for $d=6$.

    \emph{Lower bounds.} For $d=2,3$, every $\mathbf{e}_i$ is a facet normal of a single staircase simplex (the facet $x_{\sigma_d}\leq1$ of $\Delta_\sigma$ with $\sigma_d=i$, see \cref{subsec:order}), hence of a sum of three (\cref{lem:persistenceaddition}). For $d=4,5,6$, \cref{thm:padovan} gives primitive facet normals of Minkowski sums of three staircase simplices with $\lVert\alpha\rVert_\infty=\mathrm{Pad}_{d-1}=\mathrm{Pad}_3\,,\mathrm{Pad}_4\,,\mathrm{Pad}_5=2,2,3$.
\end{proof}

For comparison, \cref{thm:inftynormofalpha} gives $\mathfrak{h}_3(5)\leq\min\{\mathfrak{B}(3),\delta_4\}=3$ and $\mathfrak{h}_3(6)\leq\min\{\mathfrak{B}(4),\delta_5\}=4$. For $d=7$, the bounds of \cref{thm:inftynormofalpha}\ref{thm:size:b} and \cref{thm:padovan} give $4=\mathrm{Pad}_6\leq \mathfrak{h}_3(7)\leq \mathfrak{B}(5)=6$.

\begin{proposition}[Four summands in dimension five]\label{prop:fourinfive}
    \begin{enumerate}[label=(\alph*),ref=(\alph*)]
        \item\label{prop:ff:a} The least dimension in which some $\alpha$ has $\kappa(\alpha)=4$ is $d=5$.
        \item\label{prop:ff:b} $\mathfrak{h}_4(4)=\mathfrak{h}_3(4)=2$, whereas $\mathfrak{h}_4(5)=3>2=\mathfrak{h}_3(5)$.
        \item\label{prop:ff:c} The vector $\alpha$ of \cref{thm:narayana} satisfies $\kappa(\alpha)=4$ for every $d\geq5$.
    \end{enumerate}
\end{proposition}

\begin{proof}
    \ref{prop:ff:a} By \cref{cor:anysum}, $\kappa(\alpha)\leq d-1\leq3$ whenever $d\leq4$ and $\kappa(\alpha)<\infty$, while $\kappa(1,-1,-1,-1,-1)=4$ by \cref{cor:anysum}.

    \ref{prop:ff:b} For $d=4$, monotonicity in the number of summands (\cref{sec:size}), \cref{cor:deltabound}, and \cref{prop:H3small} give $2=\mathfrak{h}_3(4)\leq \mathfrak{h}_4(4)\leq\delta_3=2$. For $d=5$, \cref{thm:inftynormofalpha}\ref{thm:size:c} and \cref{cor:deltabound} give $\mathfrak{h}_4(5)\leq\min\{12^{(5-1)/4},\delta_4\}=3$. For the lower bound, let $\alpha$ be the vector of \cref{thm:narayana} for $d=5$, that is, $\alpha=(-1,1,-1,2,-3)$, since $\mathcal{N}_1\,,\ldots,\mathcal{N}_5=1,1,1,2,3$. Its positive entries are $\alpha_2=1$ and $\alpha_4=2$, so a nonempty $B\subset[5]$ with $\alpha^\top\mathbf{e}_B=0$ has $B\cap\{2,4\}$ of sum $1$, $2$, or $3$, matched by a single $-1$, by $-1-1$, or by $-3$; hence the subsets with $\alpha^\top\mathbf{e}_B=0$ are exactly
    \[
    B_1=\{1,2\},\qquad B_2=\{2,3\},\qquad B_3=\{1,3,4\},\qquad B_4=\{2,4,5\}.
    \]
    A vector $x$ orthogonal to $\mathbf{e}_{B_1}\,,\ldots,\mathbf{e}_{B_4}$ satisfies $x_2=-x_1$, $x_3=x_1$, $x_4=-2x_1$, $x_5=3x_1$\,, so $x\in\Span\{\alpha\}$; hence the four indicator vectors are linearly independent, and with the four classes $\{1\},\{2\},\{3\},\{4\}$, \cref{thm:facetnormalcharacterization} gives $\kappa(\alpha)\leq4$. As $\alpha$ is primitive with $\lVert\alpha\rVert_\infty=3$, we get $\mathfrak{h}_4(5)\geq3$. (The four subsets pairwise intersect, so $\Gamma(B_1,\ldots,B_4)$ is complete and $\kappa(\alpha)=4$ by \cref{cor:anysum}; this also follows from \ref{prop:ff:c}.)

        \ref{prop:ff:c} By \cref{thm:narayana}, $\kappa(\alpha)\leq4$ and $\lVert\alpha\rVert_\infty=\mathcal{N}_d$\,, and by \cref{cor:separation}\ref{cor:sep:c}, $\kappa(\alpha)=4$ for $d\geq9$. For $d=5$ and $d=6$, \cref{prop:H3small} gives $\mathcal{N}_5=3>2=\mathfrak{h}_3(5)$ and $\mathcal{N}_6=4>3=\mathfrak{h}_3(6)$, so $\alpha$ is not a facet normal of any Minkowski sum of at most three staircase simplices, and $\kappa(\alpha)=4$.

    It remains to exclude three summands for $d=7$ and $d=8$. Call a nonempty subset $B\subset[d]$ with $\alpha^\top\mathbf{e}_B=0$ an \emph{atom} if it has no nonempty proper subset with zero $\alpha$-sum. Every nonempty subset with zero $\alpha$-sum is a disjoint union of atoms: remove an atom contained in it and iterate on the remainder, which again has zero $\alpha$-sum. Hence, if $\kappa(\alpha)\leq3$, then refining the subsets $B_1\,,\ldots,B_{d-1}$ of \cref{thm:facetnormalcharacterization} into atoms (compare the proof of \cref{lem:matchingreduction}) yields three families of pairwise disjoint atoms whose indicator vectors span $\alpha^\perp$ and therefore have rank $d-1$. We show that this is impossible.

    For $d=7$, $\alpha=(-1,1,-1,2,-3,4,-6)$, and the atoms are
    \begin{gather*}
    A_1=\{1,2\},\qquad A_2=\{2,3\},\qquad A_3=\{1,3,4\},\qquad A_4=\{2,4,5\},\\
    A_5=\{1,5,6\},\qquad A_6=\{3,5,6\},\qquad A_7=\{4,6,7\}.
    \end{gather*}
    Indeed, a zero-sum subset containing $7$ must contain $4$ and $6$ to offset $-6$, and $A_7$ is then zero-sum; one containing $6$ but not $7$ must offset $4$ by $-3$ together with one of the entries $-1$, giving $A_5$ or $A_6$\,; and on $[5]$ the atoms are $A_1\,,\ldots,A_4$ by part \ref{prop:ff:b}. The only two-element atoms, $A_1$ and $A_2$\,, intersect, so three pairwise disjoint atoms would use at least $2+3+3=8>7$ coordinates; thus every family of pairwise disjoint atoms has at most two members. Each of $A_3$ and $A_4$ meets every other atom. If one of the three families contains $A_3$ or $A_4$\,, that family is a singleton, and the three families contain at most $1+2+2=5$ distinct atoms; otherwise only the five atoms $A_1,A_2,A_5,A_6,A_7$ are available. In either case the rank is at most $5<6=d-1$.

    For $d=8$, $\alpha=(-1,1,-1,2,-3,4,-6,9)$. Again $A_1$ and $A_2$ are the only two-element atoms, because $\pm1$ are the only entries that occur with both signs, and they intersect. Three pairwise disjoint atoms would therefore have sizes $2,3,3$ and cover all eight coordinates, which is impossible because the entries of $\alpha$ sum to $5\neq0$. Hence every family of pairwise disjoint atoms has at most two members, three families contain at most six atoms, and the rank is at most $6<7=d-1$.
\end{proof}

Part \ref{prop:ff:c} does not determine $\mathfrak{h}_3(7)$ and $\mathfrak{h}_3(8)$: it shows that one particular vector of height $\mathcal{N}_7=6$, respectively $\mathcal{N}_8=9$, is not a three-summand facet normal, while the bounds $4\leq \mathfrak{h}_3(7)\leq6$ and $\mathfrak{h}_3(8)\leq9$ stand.

\section{Normals with entries in \texorpdfstring{$\{-1,0,1\}$}{\{-1,0,1\}}}\label{sec:kappa}
Sums of at most two staircase simplices have only primitive facet normals in $\{-1,0,1\}^d$ (\cref{thm:inftynormofalpha}\ref{thm:size:a}). In this section we determine $\kappa(\alpha)$ for every such vector; as a consequence, we characterize the facet normals of sums of two staircase simplices (\cref{cor:two}).

\begin{theorem}\label{thm:kappa}
Let $d\geq2$, let $\alpha\in\{-1,0,1\}^d\setminus\{0\}$, and let $p_+:=|\{i\in[d]:\alpha_i=1\}|$ and $p_-:=|\{i\in[d]:\alpha_i=-1\}|$.
\begin{enumerate}[label=(\alph*),ref=(\alph*)]
    \item\label{thm:kappa:a} If $p_+\,,p_-\geq1$, then, with $p_{\min}:=\min\{p_+\,,p_-\}$ and $p_{\max}:=\max\{p_+\,,p_-\}$,
    \[
    \kappa(\alpha)=\left\lceil\frac{p_++p_--1}{p_{\min}}\right\rceil=1+\left\lceil\frac{p_{\max}-1}{p_{\min}}\right\rceil.
    \]
    \item\label{thm:kappa:b} If $p_++p_-=1$, that is, $\alpha=\pm\mathbf{e}_i$ for some $i\in[d]$, then $\kappa(\alpha)=1$.
    \item\label{thm:kappa:c} Otherwise, that is, if $\alpha$ has at least two nonzero entries, all of the same sign, then $\kappa(\alpha)=\infty$.
\end{enumerate}
\end{theorem}

In particular, in case \ref{thm:kappa:a} we have $\kappa(\alpha)=1$ if and only if $p_{\max}=1$, that is, $\alpha=\mathbf{e}_i-\mathbf{e}_j$\,, and $\kappa(\alpha)=2$ if and only if $2\leq p_{\max}\leq p_{\min}+1$.

\begin{proof}
    The two expressions in \ref{thm:kappa:a} agree because $(p_{\min}+p_{\max}-1)/p_{\min}=1+(p_{\max}-1)/p_{\min}$\,. Part \ref{thm:kappa:a} is the case $a=b=1$ of \cref{thm:twomagnitude} in \cref{sec:twomagnitude}, combined with the zero-padding lemma (\cref{lem:zeropadding}); we carry out this specialization in \cref{rem:specialcases}(i). The proof of \cref{thm:twomagnitude} does not use \cref{thm:kappa} or \cref{cor:two}.

    \smallskip
    \emph{\ref{thm:kappa:b}.} If $\alpha=\mathbf{e}_i$\,, choose a permutation $\sigma$ with $\sigma_d=i$; then $x_i=x_{\sigma_d}\leq1$ describes a facet of $\Delta_\sigma$ (\cref{subsec:order}), with normal $\mathbf{e}_i$\,. If $\alpha=-\mathbf{e}_i$\,, choose $\sigma$ with $\sigma_1=i$; then $-x_i=-x_{\sigma_1}\leq0$ describes a facet of $\Delta_\sigma$\,, with normal $-\mathbf{e}_i$\,. Hence $\kappa(\alpha)=1$.

    \smallskip
    \emph{\ref{thm:kappa:c}.} Let $S$ be the support of $\alpha$ and $Z:=[d]\setminus S$. If $\alpha$ has at least two nonzero entries, all of the same sign, then every subset $B$ with $\alpha^\top\mathbf{e}_B=0$ satisfies $B\subset Z$, because $\alpha^\top\mathbf{e}_B$ is a sum of $|B\cap S|$ nonzero terms of one sign; so $\mathbf{e}_B\in\Span\{\mathbf{e}_i:i\in Z\}$, a space of dimension $|Z|\leq d-2$. Hence no $d-1$ such vectors are linearly independent, conditions \ref{thm:fnc:a} and \ref{thm:fnc:b} cannot hold together, and by \cref{thm:facetnormalcharacterization} no Minkowski sum of staircase simplices has facet normal $\alpha$: $\kappa(\alpha)=\infty$.
\end{proof}

\begin{corollary}\label{cor:two}
Let $d\geq2$. A primitive vector $\alpha\in\mathbb{Z}^d$ is a facet normal of the sum of some two staircase simplices in $\mathbb{R}^d$ if and only if $\alpha\in\{-1,0,1\}^d\setminus\{0\}$ and $\bigl|\sum_{i\in[d]}\alpha_i\bigr|\leq 1$.
\end{corollary}

\begin{proof}
    By \cref{thm:inftynormofalpha}\ref{thm:size:a}, and by monotonicity in the number of summands, the primitive facet normals of sums of two staircase simplices are exactly the vectors $\alpha\in\{-1,0,1\}^d\setminus\{0\}$ with $\kappa(\alpha)\leq2$. Let $p_+$ and $p_-$ be as in \cref{thm:kappa}, so that $\sum_{i\in[d]}\alpha_i=p_+-p_-$\,. If $p_+\,,p_-\geq1$, \cref{thm:kappa}\ref{thm:kappa:a} gives $\kappa(\alpha)=1+\lceil(p_{\max}-1)/p_{\min}\rceil$, which is at most $2$ if and only if $p_{\max}-1\leq p_{\min}$\,, that is, $|p_+-p_-|\leq1$. If all nonzero entries of $\alpha$ have the same sign, then $\bigl|\sum_{i\in[d]}\alpha_i\bigr|=p_++p_-\geq1$, and \cref{thm:kappa}\ref{thm:kappa:b},\ref{thm:kappa:c} give $\kappa(\alpha)\leq2$ if and only if $p_++p_-=1$. In both cases, $\kappa(\alpha)\leq2$ if and only if $\bigl|\sum_{i\in[d]}\alpha_i\bigr|\leq1$.
\end{proof}

With $\hat\alpha:=\bigl(-\sum_{i}\alpha_i\,,\alpha_1\,,\ldots,\alpha_d\bigr)$, so that $\alpha^\top x=\hat\alpha^\top w$ under $x_i=w_i-w_0$\,, the condition of \cref{cor:two} says exactly that $\hat\alpha\in\{-1,0,1\}^{d+1}$. In this type-$A$ form, \cref{thm:inftynormofalpha}\ref{thm:size:a} can also be derived from the alternating-cycle description of intersections of root cones \cite[Proposition~3.6]{early2025when}; our determinant argument proves it directly, and \cref{thm:kappa} also determines $\kappa$.

\section{Two-magnitude normals}\label{sec:twomagnitude}

In this section we treat the vectors with entries $a$ and $-b$ for coprime positive integers $a$ and $b$, and determine $\kappa$ exactly; the case $a=b=1$ gives \cref{thm:kappa}\ref{thm:kappa:a}. Throughout, $a$ and $b$ are coprime positive integers, $p,q\geq1$, $d:=p+q$, and $\alpha\in\mathbb{Z}^d$ has $p$ entries equal to $a$, indexed by $S_+$\,, and $q$ entries equal to $-b$, indexed by $S_-$ (the value of $\kappa$ does not depend on how the entries are arranged, because permuting coordinates maps staircase simplices to staircase simplices). A subset $B\subset[d]$ has $\alpha^\top\mathbf{e}_B=a|B\cap S_+|-b|B\cap S_-|$; if this is zero, then $b$ divides $a|B\cap S_+|$, hence $|B\cap S_+|$, because $\gcd(a,b)=1$. So the nonempty subsets $B$ with $\alpha^\top\mathbf{e}_B=0$ are exactly those with
\[
|B\cap S_+|=bk\quad\text{and}\quad|B\cap S_-|=ak\qquad\text{for some integer }k\geq1;
\]
we call $B$ an \emph{atom} if $k=1$. Put
\[
\mu:=\min\{\lfloor p/b\rfloor,\lfloor q/a\rfloor\},\qquad
U:=1+(p-1)\,\mathbf{1}_{\{p>b\mu\}}+(q-1)\,\mathbf{1}_{\{q>a\mu\}}\,.
\]
Every atom has $b$ elements in $S_+$ and $a$ elements in $S_-$\,, so a set of pairwise disjoint atoms has at most $\min\{\lfloor p/b\rfloor,\lfloor q/a\rfloor\}=\mu$ members. We call the positive side \emph{saturated} if $p=b\mu$, and the negative side \emph{saturated} if $q=a\mu$; thus
\[
U=\begin{cases}
1, & \text{both sides saturated},\\
q, & \text{only the positive side saturated},\\
p, & \text{only the negative side saturated},\\
d-1, & \text{neither side saturated}.
\end{cases}
\]

A \emph{full matching} is a set of $\mu$ pairwise disjoint atoms; full matchings exist whenever $\mu\geq1$. For an integer $t\geq1$, let $\rho_t$ be the largest rank of the set of indicator vectors of the atoms in a union of $t$ full matchings, repetitions among the matchings allowed. The following theorem determines $\kappa$ exactly for every two-magnitude vector.

\begin{theorem}[Exact summand count]\label{thm:twomagnitude}
    Let $\alpha$ be as above. If $\mu=0$, then $\kappa(\alpha)=\infty$. If $\mu\geq1$, then
    \begin{equation}\label{eq:exactkappa}
    \kappa(\alpha)=\min\bigl\{t\geq1:\ t(\mu-1)+\min\{t,U\}\geq d-1\bigr\},
    \end{equation}
    where the minimum of the empty set is $\infty$. In particular, $\kappa(\alpha)<\infty$ if and only if
    \[
    p\geq b,\qquad q\geq a,\qquad(p>b\ \text{or}\ b=1),\qquad\text{and}\qquad(q>a\ \text{or}\ a=1),
    \]
    and in every finite case with $\mu=1$, $\kappa(\alpha)=d-1$. For $\mu\geq2$,
    \[
    \kappa(\alpha)=
    \begin{cases}
    \left\lceil\dfrac{p+q-2}{\mu-1}\right\rceil, & p=b\mu,\ q=a\mu,\\[2.2ex]
    \max\left\{\left\lceil\dfrac{p+q-1}{\mu}\right\rceil,\ \left\lceil\dfrac{p-1}{\mu-1}\right\rceil\right\}, & p=b\mu,\ q>a\mu,\\[2.2ex]
    \max\left\{\left\lceil\dfrac{p+q-1}{\mu}\right\rceil,\ \left\lceil\dfrac{q-1}{\mu-1}\right\rceil\right\}, & p>b\mu,\ q=a\mu,\\[2.2ex]
    \left\lceil\dfrac{p+q-1}{\mu}\right\rceil, & p>b\mu,\ q>a\mu.
    \end{cases}
    \]
\end{theorem}

We prove \cref{thm:twomagnitude} in three steps: a reduction to full matchings (\cref{lem:matchingreduction}), an exact formula for $\rho_t$ (\cref{thm:matchingrank}), and a case analysis. We then illustrate the formula (\cref{rem:specialcases,ex:generic}), determine the maximum dimension of the exposed face of $\alpha$ (\cref{thm:exposedface}), and extend the result to vectors with zero entries (\cref{cor:twovalues}).

\begin{lemma}[Reduction to full matchings]\label{lem:matchingreduction}
    Let $\mu\geq1$ and $t\geq1$. Then $\kappa(\alpha)\leq t$ if and only if $\rho_t=d-1$.
\end{lemma}

\begin{proof}
    Suppose $\kappa(\alpha)\leq t$, and take subsets $B_1\,,\ldots,B_{d-1}$ as in statement \ref{thm:fnc:2} of \cref{thm:facetnormalcharacterization}, with classes $J_1\,,\ldots,J_t$ in \ref{thm:fnc:c}, some possibly empty. By \ref{thm:fnc:a} and \ref{thm:fnc:b}, the indicator vectors $\mathbf{e}_{B_j}$ span $\alpha^\perp$. We modify the family; every modification keeps all indicator vectors in $\alpha^\perp$ (atoms have $\alpha$-sum zero) and does not shrink their span, so the span remains $\alpha^\perp$. \emph{Refine:} a subset $B_j$ with $|B_j\cap S_+|=bk$ and $|B_j\cap S_-|=ak$ partitions into $k$ atoms (split $B_j\cap S_+$ into $k$ sets of size $b$ and $B_j\cap S_-$ into $k$ sets of size $a$, and pair them up); replace each $B_j$ by the atoms of such a partition, all placed in the class of $B_j$\,. Within each class the atoms remain pairwise disjoint, because atoms of one $B_j$ are disjoint by construction and atoms of distinct $B_j\,,B_{j'}$ in the same class lie in the disjoint sets $B_j\,,B_{j'}$\,; and the span can only grow, because $\mathbf{e}_{B_j}$ is the sum of the indicator vectors of its atoms. \emph{Extend:} each class is now a set of $m\leq \mu$ pairwise disjoint atoms; if $m<\mu$, then its atoms use at most $(\mu-1)b\leq p-b$ elements of $S_+$ and at most $(\mu-1)a\leq q-a$ elements of $S_-$ (because $b\mu\leq p$ and $a\mu\leq q$), so a further atom disjoint from them can be added; repeat until every class is a full matching. The result is a union of $t$ full matchings whose atom indicators span $\alpha^\perp$, so $\rho_t=d-1$.

    Conversely, if the atom indicators of a union of $t$ full matchings have rank $d-1$, select $d-1$ linearly independent ones and assign each selected atom to one of the $t$ matchings containing it, the matchings serving as the classes. Atoms have $\alpha$-sum zero, which is \ref{thm:fnc:a}; the selected indicators are linearly independent, which is \ref{thm:fnc:b}; and the atoms in one matching are pairwise disjoint, which is \ref{thm:fnc:c} with $n+1=t$. Hence $\kappa(\alpha)\leq t$.
\end{proof}

\begin{theorem}[Exact matching rank]\label{thm:matchingrank}
    Suppose $\mu\geq1$. Then, for every integer $t\geq1$,
    \begin{equation}\label{eq:matchingrank}
    \rho_t=\min\bigl\{\,d-1,\ t\mu,\ t(\mu-1)+U\,\bigr\}.
    \end{equation}
\end{theorem}

\begin{proof}
    Write $r:=\mu-1$.

    \smallskip
    \emph{Upper bound.} The atom indicators lie in $\alpha^\perp$, giving rank at most $d-1$; and a union of $t$ full matchings contains at most $t\mu$ distinct atoms, because each full matching consists of exactly $\mu$ atoms (fewer than $t\mu$ occur when two of the matchings share an atom, in particular when a matching is repeated). If $u_1\,,\ldots,u_\mu$ are the atom indicators of one full matching, then they are spanned by the $\mu-1$ differences $u_\ell-u_\mu$ ($1\leq\ell\leq \mu-1$) together with their sum $u_1+\cdots+u_\mu=\mathbf{e}_X+\mathbf{e}_Y$\,, where $X\subset S_+$\,, $|X|=b\mu$, and $Y\subset S_-$\,, $|Y|=a\mu$, are the unions of the positive and the negative parts of the atoms. The span of all vectors $\mathbf{e}_X+\mathbf{e}_Y$ of this form has dimension at most $U$. Indeed, fix one full matching, with sum $\mathbf{e}_{X_0}+\mathbf{e}_{Y_0}$\,. For any full matching, $(\mathbf{e}_X+\mathbf{e}_Y)-(\mathbf{e}_{X_0}+\mathbf{e}_{Y_0})=(\mathbf{e}_X-\mathbf{e}_{X_0})+(\mathbf{e}_Y-\mathbf{e}_{Y_0})$, where $\mathbf{e}_X-\mathbf{e}_{X_0}$ has support in $S_+$ and coordinate sum $|X|-|X_0|=0$, and is zero if the positive side is saturated, because then $X=X_0=S_+$; likewise for $\mathbf{e}_Y-\mathbf{e}_{Y_0}$\,. Hence every such sum lies in the span of $\mathbf{e}_{X_0}+\mathbf{e}_{Y_0}$\,, the coordinate-sum-zero vectors supported on $S_+$ if the positive side is unsaturated, and the coordinate-sum-zero vectors supported on $S_-$ if the negative side is unsaturated, a space of dimension $1+(p-1)\,\mathbf{1}_{\{p>b\mu\}}+(q-1)\,\mathbf{1}_{\{q>a\mu\}}=U$. In a union of $t$ full matchings, the $t$ sums $\mathbf{e}_X+\mathbf{e}_Y$ therefore span at most $\min\{t,U\}$ dimensions, and the $t(\mu-1)$ differences span at most $t(\mu-1)$. Hence the rank is at most $t(\mu-1)+\min\{t,U\}=\min\{t\mu,\ t(\mu-1)+U\}$, and together with the first bound, $\rho_t\leq\min\{d-1,\ t\mu,\ t(\mu-1)+U\}$.

    \smallskip
    \emph{One matching and its symmetries.} For the lower bound it suffices to find $t$ coordinate permutations $g_1\,,\ldots,g_t$ preserving the two sides such that $\dim(g_1W+\cdots+g_tW)\geq\min\{d-1,\ t\mu,\ t(\mu-1)+U\}$, where $W$ is the span of the atom indicators of one fixed full matching. We find them by exhibiting a nonzero product $\eta_1\wedge\cdots\wedge\eta_t$ in which each $\eta_\ell$ lies in the span of the exterior powers of the translates $gW$; representation theory identifies these spans. We work over $\mathbb{C}$, which does not change the rank of a real matrix. Put $E:=\{v\in\mathbb{C}^d:\alpha^\top v=0\}$. Let $V_+$ and $V_-$ be the coordinate-sum-zero subspaces supported on $S_+$ and on $S_-$\,, of dimensions $p-1$ and $q-1$, and put
    \[
    z:=\frac{b}{p}\,\mathbf{e}_{S_+}+\frac{a}{q}\,\mathbf{e}_{S_-}\,,\qquad T:=\mathbb{C}z.
    \]
    Then $\alpha^\top z=ab-ba=0$. The spaces $V_+$ and $V_-$ have disjoint supports, and every vector in their sum has coordinate sum zero on each side, whereas $z$ has coordinate sums $b$ and $a$ on the two sides; so the sum of $V_+$, $V_-$, and $T$ is direct, and comparing dimensions gives $E=V_+\oplus V_-\oplus T$.
    The group $G:=\mathfrak{S}_{S_+}\times\mathfrak{S}_{S_-}$ of coordinate permutations preserving the two sides acts on $E$ and fixes $z$.

    Fix a full matching with atom indicators $u_1\,,\ldots,u_\mu$\,, and write $u_\ell=\mathbf{e}_{A_\ell}+\mathbf{e}_{C_\ell}$ with $A_\ell\subset S_+$\,, $|A_\ell|=b$, and $C_\ell\subset S_-$\,, $|C_\ell|=a$, where the sets $A_\ell$ are pairwise disjoint, as are the sets $C_\ell$\,. Put $X:=\bigcup_{\ell=1}^\mu A_\ell$\,, $Y:=\bigcup_{\ell=1}^\mu C_\ell$\,, and $W:=\Span\{u_1\,,\ldots,u_\mu\}$. The atom indicators have pairwise disjoint nonempty supports, so $\dim W=\mu$. Moreover, every full matching is a $G$-translate of this one: after pairing its atoms with $u_1\,,\ldots,u_\mu$\,, choose permutations of $S_+$ and of $S_-$ carrying the corresponding positive parts, negative parts, and unused coordinates onto one another.

    For $1\leq\ell\leq r$, define
    \[
    x_\ell:=\mathbf{e}_{A_\ell}-\mathbf{e}_{A_\mu}\in V_+\,,\qquad y_\ell:=\mathbf{e}_{C_\ell}-\mathbf{e}_{C_\mu}\in V_-\,,
    \]
    and define
    \[
    x_0:=\frac1\mu\,\mathbf{e}_X-\frac bp\,\mathbf{e}_{S_+}\in V_+\,,\qquad y_0:=\frac1\mu\,\mathbf{e}_Y-\frac aq\,\mathbf{e}_{S_-}\in V_-\,.
    \]
    Then $u_\ell-u_\mu=x_\ell+y_\ell$ for $1\leq\ell\leq r$, and $\frac1\mu\sum_{\ell=1}^\mu u_\ell=\frac1\mu(\mathbf{e}_X+\mathbf{e}_Y)=z+x_0+y_0$\,. The differences $u_\ell-u_\mu$ together with the average of the $u_\ell$ span $W$, hence
    \begin{equation}\label{eq:Wbasis}
    W=\Span\{\,x_1+y_1\,,\ \ldots,\ x_r+y_r\,,\ z+x_0+y_0\,\}.
    \end{equation}
    The vectors $x_1\,,\ldots,x_r$ are linearly independent: in a vanishing linear combination, a coordinate in $A_\ell$\,, $1\leq\ell\leq r$, occurs only in $x_\ell$ and forces its coefficient to vanish. The same argument shows that $y_1\,,\ldots,y_r$ are linearly independent. If the positive side is saturated, then $X=S_+$ and $p=b\mu$, so $x_0=0$. Otherwise $S_+\setminus X\neq\varnothing$, and $x_0$ takes the nonzero value $-b/p$ there, whereas every $x_\ell$\,, $1\leq\ell\leq r$, vanishes there; thus $x_0\notin\Span\{x_1\,,\ldots,x_r\}$ when the positive side is unsaturated. Likewise, $y_0=0$ when the negative side is saturated, and $y_0\notin\Span\{y_1\,,\ldots,y_r\}$ when it is unsaturated.

    \smallskip
    \emph{Exterior components.} The exterior powers $\Lambda^iV_+$\,, $0\leq i\leq p-1$, of the standard representation of $\mathfrak{S}_{S_+}$ are irreducible and pairwise inequivalent (they are the hook representations; see \cite[Exercise~4.6]{FultonHarris}), and likewise for $V_-$\,. As $V_+$\,, $V_-$\,, and $T$ are $G$-invariant and $\dim T=1$, for each $k$
    \begin{equation}\label{eq:lambdadecomp}
    \Lambda^kE=\bigoplus_{\substack{i+j+h=k\\ 0\leq i\leq p-1,\ 0\leq j\leq q-1\\ h\in\{0,1\}}}\Lambda^iV_+\otimes\Lambda^jV_-\otimes\Lambda^hT
    \end{equation}
    is a decomposition into irreducible $G$-representations, namely external tensor products of irreducible representations of $\mathfrak{S}_{S_+}$ and $\mathfrak{S}_{S_-}$\,, tensored with the trivial representation $\Lambda^hT$; they are pairwise inequivalent, because the isomorphism type determines $(i,j)$ and hence $h=k-i-j$. So the $G$-invariant subspace $\mathcal{L}_k:=\Span_{g\in G}\Lambda^k(gW)$ of $\Lambda^kE$ is a sum of some of these components, and it contains every component onto which $\Lambda^kW\subset\mathcal{L}_k$ has a nonzero projection.

    We next identify components with nonzero projection, using the spanning set \eqref{eq:Wbasis} of $W$. Extend $x_1\,,\ldots,x_r$ to a basis of $V_+$\,, including $x_0$ when the positive side is unsaturated, and extend $y_1\,,\ldots,y_r$ to a basis of $V_-$\,, including $y_0$ when the negative side is unsaturated. The triples $(i,j,h)$ in (i)--(iii) lie in the index range of \eqref{eq:lambdadecomp}, because $\mu\leq\min\{p,q\}$, with $\mu\leq p-1$ (resp.\ $\mu\leq q-1$) when the positive (resp.\ negative) side is unsaturated.
    \begin{enumerate}[label=(\roman*)]
        \item Suppose $h=0$ and $i+j\leq r$. Choose $K\subset[r]$ with $|K|=i+j$, and consider $\bigwedge_{\ell\in K}(x_\ell+y_\ell)\in\Lambda^{i+j}W$. Since every $x_\ell+y_\ell$ lies in $V_+\oplus V_-$\,, expanding the product shows that its projection onto $\Lambda^iV_+\otimes\Lambda^jV_-$ equals
        \[
        \sum_{I\subset K,\ |I|=i}\pm\Bigl(\bigwedge_{\ell\in I}x_\ell\Bigr)\otimes\Bigl(\bigwedge_{\ell\in K\setminus I}y_\ell\Bigr).
        \]
        The summands are distinct basis monomials, each with coefficient $1$ or $-1$, so the projection is nonzero.
        \item Suppose $h=1$ and $i+j\leq r$. Wedge the exterior product in (i) with $z+x_0+y_0$\,; the result lies in $\Lambda^{i+j+1}W$ because $z+x_0+y_0\in W$ by \eqref{eq:Wbasis}. Since $x_0\in V_+$ and $y_0\in V_-$\,, only the term with $z$ contributes to components with $h=1$, so the projection onto $\Lambda^iV_+\otimes\Lambda^jV_-\otimes T$ is the projection from (i) wedged with $z$. It is nonzero, because $z\notin V_+\oplus V_-$\,, so wedging with $z$ is injective on $\Lambda(V_+\oplus V_-)$.
        \item Suppose $h=0$ and $i+j=\mu=r+1$. If the positive side is unsaturated and $i\geq1$, expand
        \[
        \Bigl(\bigwedge_{\ell=1}^r(x_\ell+y_\ell)\Bigr)\wedge(z+x_0+y_0)\in\Lambda^\mu W.
        \]
        Choose $I\subset[r]$ with $|I|=i-1$ and put $J:=[r]\setminus I$, so that $|J|=j$. In the expansion of the product, the basis monomial $\bigl(x_0\wedge\bigwedge_{\ell\in I}x_\ell\bigr)\otimes\bigl(\bigwedge_{\ell\in J}y_\ell\bigr)$ of $\Lambda^iV_+\otimes\Lambda^jV_-$ arises from exactly one term, so its coefficient in the projection onto $\Lambda^iV_+\otimes\Lambda^jV_-$ is $1$ or $-1$. No other term produces it: another choice of $I$ gives a different monomial, terms using $y_0$ contain no $x_0$\,, and terms using $z$ have $h=1$. The projection is therefore nonzero. If instead the negative side is unsaturated and $j\geq1$, the same argument uses $y_0$\,.
    \end{enumerate}

    \smallskip
    \emph{Allocation of exterior degrees.} Call a triple $(i,j,h)$ \emph{available} if $h=0$ and $i+j\leq r$, as in (i); or $h=1$ and $i+j\leq r$, as in (ii); or $h=0$, $i+j=\mu$, and either $i\geq1$ with the positive side unsaturated or $j\geq1$ with the negative side unsaturated, as in (iii). By (i)--(iii), $\Lambda^{i+j+h}W$ has a nonzero projection onto the component $\Lambda^iV_+\otimes\Lambda^jV_-\otimes\Lambda^hT$ of every available triple. Set $k:=\min\{d-1,\ t\mu,\ tr+U\}$, the target rank. We construct available triples $(i_\ell\,,j_\ell\,,h_\ell)$, $1\leq\ell\leq t$, such that
    \begin{gather*}
    h_1=1,\qquad h_\ell=0\quad(\ell\geq2),\qquad \sum_{\ell=1}^t(i_\ell+j_\ell)=k-1,\\
    \sum_{\ell=1}^ti_\ell\leq p-1,\qquad \sum_{\ell=1}^tj_\ell\leq q-1.
    \end{gather*}
    Call a degree on the positive (resp.\ negative) side \emph{free} if that side is unsaturated, and put $e:=\max\{0,\ k-1-tr\}$; since $k\leq t\mu$, we have $e\leq t-1$. Choose totals $I\leq p-1$ and $J\leq q-1$ with $I+J=k-1$, which is possible because $k-1\leq d-2$, so that the number $c:=I\,\mathbf{1}_{\{p>b\mu\}}+J\,\mathbf{1}_{\{q>a\mu\}}$ of free degrees is as large as possible. Then $c\geq e$. Indeed, if both sides are saturated, then $U=1$, so $k\leq tr+1$ and $e=0$; if both are unsaturated, then $c=k-1\geq e$; and if only the negative side is unsaturated, then $c=\min\{q-1,\ k-1\}$, and $q-1\geq e$ because $k\leq tr+U=tr+q$; the case in which only the positive side is unsaturated is symmetric.

    Now distribute these $k-1$ degrees over $t$ slots: each of the slots $2,\ldots,e+1$ first receives one free degree, and the remaining $k-1-e$ degrees are then placed arbitrarily so that every slot receives at most $r$ additional degrees, which is possible because $tr\geq k-1-e$. Let $i_\ell$ and $j_\ell$ be the numbers of positive and negative degrees in slot $\ell$. Then $i_1+j_1\leq r$, so $(i_1\,,j_1\,,1)$ is available by (ii). For $\ell\geq2$, either $i_\ell+j_\ell\leq r$, and $(i_\ell\,,j_\ell\,,0)$ is available by (i), or $i_\ell+j_\ell=\mu$ and slot $\ell$ contains a free degree, and $(i_\ell\,,j_\ell\,,0)$ is available by (iii). Finally, $\sum_\ell i_\ell=I\leq p-1$ and $\sum_\ell j_\ell=J\leq q-1$. This argument includes the case $\mu=1$, in which $r=0$.

    \smallskip
    \emph{From orbit spans to matchings.} Put $k_\ell:=i_\ell+j_\ell+h_\ell$\,. Fix arbitrary bases of $V_+$ and of $V_-$\,. For each $\ell$, choose $i_\ell$ positive-side basis vectors and $j_\ell$ negative-side basis vectors, using disjoint subsets of each basis for different $\ell$; this is possible by the two overall degree bounds. Let $\eta_\ell$ be their exterior product, also including $z$ when $\ell=1$. Then $\eta_\ell$ is a nonzero element of the component indexed by $(i_\ell\,,j_\ell\,,h_\ell)$. Since this triple is available, that component is contained in $\mathcal{L}_{k_\ell}$\,, so $\eta_\ell\in\mathcal{L}_{k_\ell}$\,. Moreover, $\eta_1\wedge\cdots\wedge\eta_t\neq0$, because all chosen basis vectors are distinct and $z$ occurs exactly once, so the product is, up to sign, a basis monomial of $\Lambda^kE$.

    By the definition of $\mathcal{L}_{k_\ell}$\,, each $\eta_\ell$ is a finite linear combination of exterior products of $k_\ell$ vectors from spaces $gW$, $g\in G$. Expanding $\eta_1\wedge\cdots\wedge\eta_t$ multilinearly, at least one term is nonzero. That term supplies $g_1\,,\ldots,g_t\in G$ and $k_\ell$ vectors from each $g_\ell W$ whose combined exterior product is nonzero. Hence these $k=\sum_\ell k_\ell$ vectors are linearly independent, and $\dim(g_1W+\cdots+g_tW)\geq k$. When $k_\ell=0$, the corresponding exterior product is $1$, and $g_\ell$ may be chosen arbitrarily. Each $g_\ell W$ is the span of the atom indicators of a full matching, so $\rho_t\geq k$. Together with the upper bound, this proves \eqref{eq:matchingrank}.
\end{proof}

\begin{proof}[Proof of \cref{thm:twomagnitude}]
    Because $\gcd(a,b)=1$, every nonempty subset $B$ with $\alpha^\top\mathbf{e}_B=0$ contains $bk$ elements of $S_+$ and $ak$ elements of $S_-$ for some $k\geq1$; so if such a subset exists, then $b\leq p$ and $a\leq q$, that is, $\mu\geq1$. Hence, if $\mu=0$, there are no such subsets, the vectors in \cref{cor:anysum} span $\{0\}\neq\alpha^\perp$, and $\kappa(\alpha)=\infty$ by \cref{cor:anysum}. Let $\mu\geq1$. By \cref{lem:matchingreduction,thm:matchingrank}, $\kappa(\alpha)\leq t$ exactly when $\min\{d-1,\,t\mu,\,t(\mu-1)+U\}=d-1$, that is, when $t(\mu-1)+\min\{t,U\}=\min\{t\mu,\,t(\mu-1)+U\}\geq d-1$; this proves \eqref{eq:exactkappa}.

    Let $\mu=1$. Then \eqref{eq:exactkappa} reads $\kappa(\alpha)=\min\{t\geq1:\ \min\{t,U\}\geq d-1\}$, and since $U\leq1+(p-1)+(q-1)=d-1$, we get $\kappa(\alpha)=d-1$ if $U=d-1$ and $\kappa(\alpha)=\infty$ otherwise. Here $U=1+(p-1)\,\mathbf{1}_{\{p>b\}}+(q-1)\,\mathbf{1}_{\{q>a\}}$\,, so $U=d-1$ if and only if ($p>b$ or $p=1$) and ($q>a$ or $q=1$); as $p\geq b\geq1$ and $q\geq a\geq1$, this is ($p>b$ or $b=1$) and ($q>a$ or $a=1$).

    Let $\mu\geq2$. Then $p\geq2b>b$ and $q\geq2a>a$, so the conditions of the finiteness criterion hold. Because $t(\mu-1)+\min\{t,U\}=\min\{t\mu,\ t(\mu-1)+U\}$, the condition in \eqref{eq:exactkappa} means $t\geq(d-1)/\mu$ and $t\geq(d-1-U)/(\mu-1)$, so
    \begin{equation}\label{eq:kappatwoterms}
    \kappa(\alpha)=\max\Bigl\{\Bigl\lceil\frac{d-1}{\mu}\Bigr\rceil,\ \Bigl\lceil\frac{d-1-U}{\mu-1}\Bigr\rceil\Bigr\}<\infty;
    \end{equation}
    the requirement $t\geq1$ is automatic because $d\geq2$. Substituting $U=1$, $q$, $p$, and $d-1$ in the four cases gives the four lines of the displayed formula. In the first case, $d-1-U=p+q-2$, and the second term dominates, because $(d-2)/(\mu-1)\geq(d-1)/\mu$ is equivalent to $d\geq \mu+1$, which holds as $d=(a+b)\mu\geq2\mu$; in the last case the second term is $0$.

    Finally, the first two conditions of the finiteness criterion, $p\geq b$ and $q\geq a$, say exactly that $\mu\geq1$, and the three preceding paragraphs show that, given $\mu\geq1$, $\kappa(\alpha)<\infty$ if and only if the last two conditions hold.
\end{proof}

\begin{remark}\label{rem:specialcases}
    We note two consequences of \cref{thm:twomagnitude}.
    \begin{enumerate}[label=(\roman*)]
        \item \emph{Entries $\pm1$; proof of \cref{thm:kappa}\ref{thm:kappa:a}.} Let $\alpha\in\{-1,0,1\}^d$ with $p_+\,,p_-\geq1$. By \cref{lem:zeropadding}, applied with $S$ the support of $\alpha$, and by permuting coordinates, $\kappa(\alpha)$ equals $\kappa$ of the vector with $p:=p_+$ entries $1$ and $q:=p_-$ entries $-1$, to which \cref{thm:twomagnitude} applies with $a=b=1$. An atom is then a pair $\{i,j\}$ with $\alpha_i=1$ and $\alpha_j=-1$, and $\mu=\min\{p,q\}=p_{\min}$\,, so we must show $\kappa(\alpha)=\lceil(p+q-1)/\mu\rceil$. By symmetry let $p\leq q$; then $p=\mu$, so the positive side is saturated, and the negative side is saturated exactly when $q=p$; in particular, the case of \cref{thm:twomagnitude} in which neither side is saturated does not occur. \Cref{thm:twomagnitude} therefore gives
        \[
        \kappa(\alpha)=
        \begin{cases}
        p+q-1, & \mu=1,\\
        \lceil(2\mu-2)/(\mu-1)\rceil=2, & \mu\geq2,\ q=p,\\
        \max\bigl\{\lceil(p+q-1)/\mu\rceil,\ \lceil(p-1)/(\mu-1)\rceil\bigr\}, & \mu\geq2,\ q>p,
        \end{cases}
        \]
        and in each case this equals $\lceil(p+q-1)/\mu\rceil$: in the second case $\lceil(2\mu-1)/\mu\rceil=2$, and in the third $\lceil(p-1)/(\mu-1)\rceil=1$ because $p=\mu$. This proves \cref{thm:kappa}\ref{thm:kappa:a}. For example, $\kappa(1,-1)=1$ and $\kappa(1,1,-1,-1,-1)=2$.
        \item \emph{A counting bound.} Let $\mu\geq1$. A staircase simplex has at most $\mu$ maximizing blocks with respect to $\alpha$, because these blocks are pairwise disjoint sets with zero $\alpha$-sum and each contains an atom; hence $\kappa(\alpha)\geq\lceil(d-1)/\mu\rceil$ by \cref{lem:dimofFMP}. If neither side is saturated, that is, $p>b\mu$ and $q>a\mu$, then \cref{thm:twomagnitude} shows that equality holds; see \cref{ex:generic} below.
    \end{enumerate}
\end{remark}

The proof of \cref{thm:matchingrank} is not constructive: it produces the matchings through a nonzero exterior form. We record one explicit optimal certificate.

\begin{example}\label{ex:generic}
    Let $\alpha:=(2,2,2,-1,-1,-1,-1,-1)\in\mathbb{Z}^8$, so that $a=2$, $b=1$, $p=3$, $q=5$, $S_+=\{1,2,3\}$, $S_-=\{4,\ldots,8\}$, and $\mu=\min\{3,2\}=2$. Neither side is saturated, so \cref{thm:twomagnitude} gives $\kappa(\alpha)=\lceil(d-1)/\mu\rceil=4$. An atom consists of one coordinate in $S_+$ and two in $S_-$\,. Consider the four full matchings
    \begin{gather*}
    \mathcal{M}_1=\bigl\{\{1,4,5\},\{2,6,7\}\bigr\},\qquad
    \mathcal{M}_2=\bigl\{\{2,4,6\},\{3,5,8\}\bigr\},\\
    \mathcal{M}_3=\bigl\{\{3,4,7\},\{1,5,6\}\bigr\},\qquad
    \mathcal{M}_4=\bigl\{\{1,7,8\},\{2,5,6\}\bigr\}.
    \end{gather*}
    If $w\in\mathbb{R}^8$ is orthogonal to the indicator vectors of all eight atoms, then comparing atoms that share a coordinate ($\{1,4,5\}$ with $\{1,5,6\}$; $\{2,4,6\}$ with $\{2,5,6\}$ and $\{2,6,7\}$; $\{3,4,7\}$ with $\{3,5,8\}$) gives in turn $w_4=w_5=w_6=w_7=w_8$\,, and then $w_1=w_2=w_3=-2w_4$\,. So $w$ is a multiple of $\alpha$, the eight indicator vectors span $\alpha^\perp$, and $\rho_4=d-1$; since the atom $\{1,7,8\}$ was not used, the other seven indicator vectors already span $\alpha^\perp$.

    These seven atoms satisfy conditions \ref{thm:fnc:a}--\ref{thm:fnc:c} of \cref{thm:facetnormalcharacterization}, with the atoms of $\mathcal{M}_k$ other than $\{1,7,8\}$ as the $k$-th class, so $\alpha$ is a facet normal of a sum of four staircase simplices and $\kappa(\alpha)\leq4$; conversely, $\kappa(\alpha)\geq\lceil7/2\rceil=4$ by \cref{rem:specialcases}(ii). Thus the four matchings form an explicit optimal certificate.
\end{example}

The matching-rank formula governs more than facets: it gives the exact maximum dimension of the exposed face of $\alpha$ over all sums of $t$ staircase simplices.

\begin{corollary}[Maximum exposed-face dimension]\label{thm:exposedface}
    Let $\alpha$, $\mu$, and $U$ be as above, and let $t\geq1$. Over all permutations $\sigma^1,\ldots,\sigma^t$ of $[d]$,
    \[
    \max_{\sigma^1,\ldots,\sigma^t}\ \dim F_{\Delta_{\sigma^1}+\cdots+\Delta_{\sigma^t}}(\alpha)
    =\begin{cases}
    0, & \mu=0,\\
    \min\bigl\{\,d-1,\ t\mu,\ t(\mu-1)+U\,\bigr\}=\rho_t\,, & \mu\geq1.
    \end{cases}
    \]
\end{corollary}

\begin{proof}
    By \cref{lem:dimofFMP}, the dimension equals the rank of the indicator vectors of the maximizing blocks of $\sigma^1,\ldots,\sigma^t$\,. If $\mu=0$, there are no nonempty zero-sum subsets (proof of \cref{thm:twomagnitude}), so the face is a point. If $\mu\geq1$, refining the blocks of each $\sigma^k$ into atoms and extending them to a full matching, as in the proof of \cref{lem:matchingreduction}, shows that the dimension is at most $\rho_t$\,. Conversely, for $t$ full matchings attaining $\rho_t$\,, the construction in the proof of \ref{thm:fnc:2} $\Rightarrow$ \ref{thm:fnc:1} of \cref{thm:facetnormalcharacterization}, which uses only \ref{thm:fnc:a} and the disjointness within each class, gives permutations $\sigma^1,\ldots,\sigma^t$ with the indicator vector of every atom of the $k$-th matching in $\dir F_{\Delta_{\sigma^k}}(\alpha)$; so the dimension $\rho_t$ is attained, and \cref{thm:matchingrank} gives its value. In particular, for $\mu\geq1$ the least codimension of the exposed face is $\max\{1,\ d-t\mu,\ d-t(\mu-1)-U\}$.
\end{proof}

Combined with the zero-padding lemma, \cref{thm:twomagnitude} determines $\kappa$ for every vector whose nonzero entries take at most two distinct values.

\begin{corollary}[At most two distinct nonzero values]\label{cor:twovalues}
    Let $d\geq2$, and let $\alpha\in\mathbb{R}^d\setminus\{0\}$ have at most two distinct nonzero entry values.
    \begin{enumerate}[label=(\alph*),ref=(\alph*)]
        \item\label{cor:twovalues:a} If the two nonzero values have opposite signs and rational ratio, write them as $ca$ and $-cb$ with $c>0$ and $a,b$ coprime positive integers, and let $p$ and $q$ be their multiplicities. Then $\kappa(\alpha)$ equals the value given by \cref{thm:twomagnitude} for these $a$, $b$, $p$, and $q$; in particular, it may be $\infty$.
        \item\label{cor:twovalues:b} If $\alpha$ has exactly one nonzero entry, then $\kappa(\alpha)=1$.
        \item\label{cor:twovalues:c} If $\alpha$ has at least two nonzero entries, all of the same sign, or if its two nonzero values have opposite signs and irrational ratio, then $\kappa(\alpha)=\infty$.
    \end{enumerate}
    These three cases are exhaustive. In particular, for $\alpha\in\{-1,0,1\}^d$ the corollary recovers \cref{thm:kappa}; see \cref{rem:specialcases}(i).
\end{corollary}

\begin{proof}
    For \ref{cor:twovalues:a}, scale $\alpha$ by $1/c$, which does not change $\kappa$ (comment after \cref{thm:facetnormalcharacterization}); by \cref{lem:zeropadding}, applied with $S$ the support of $\alpha$, and a permutation of coordinates, \cref{thm:twomagnitude} applies. Part \ref{cor:twovalues:b} is \cref{thm:kappa}\ref{thm:kappa:b} after scaling. For \ref{cor:twovalues:c}, if some subsets satisfy \ref{thm:fnc:a} and \ref{thm:fnc:b} of \cref{thm:facetnormalcharacterization} for $\alpha$, then $\alpha$ is a multiple of an integer vector by \cref{lem:crossproduct}, so its nonzero entries have rational ratios; and if the nonzero entries of $\alpha$ all have the same sign, then every zero-sum subset avoids the support of $\alpha$, so the zero-sum indicator vectors span a space of dimension at most $d-2$. In both cases \ref{thm:fnc:2} fails for every $n$. The three cases are exhaustive.
\end{proof}

\section{Order polytopes}\label{sec:orderpolytopes}
We return to the setting of the first paper. Let $P$ be a poset on $[d]$ with strict order relation $\prec_P$\,, and let $\mathcal L(P)=\{\sigma^0,\ldots,\sigma^n\}$ be its set of linear extensions. The staircase simplices $\Delta_\sigma$\,, $\sigma\in\mathcal L(P)$, form Stanley's triangulation of the order polytope $\OP=\{x\in[0,1]^d: x_i\leq x_j \text{ for } i\prec_P j\}$ (see \cite{Stanley1986}), and $M(P):=\sum_{k\in N_0}\Delta_{\sigma^k}$ is the Minkowski sum studied in \cite{QL_orderpolytope}. By \CP{\CPbijection}, the facet-normal rays of $M(P)$ are in bijection with the vertical facets of the Cayley embedding $\D$ of the simplices of Stanley's triangulation, and the facet of $\D$ corresponding to $\alpha$ is described by the full optimal big-M lifting of the supporting inequality with normal $\alpha$ of any one simplex \CP{\CPcayley}; so everything below about facet normals of $M(P)$ is a statement about the facets of $\D$.

\label{rem:facetcharMP}%
\emph{The exact test.} \Cref{thm:facetnormalcharacterization} applies to $M(P)$ with $n+1=|\mathcal L(P)|$. The implication \ref{thm:fnc:1} $\Rightarrow$ \ref{thm:fnc:2} gives necessary conditions for $\alpha$ to be a facet normal of $M(P)$, while \ref{thm:fnc:2} $\Rightarrow$ \ref{thm:fnc:1} gives a sufficient condition only when its construction can be carried out inside $\mathcal L(P)$. The exact test remains \cref{lem:dimofFMP}, that is, \CP{\CPdimofFMP}, applied with $\{\sigma^0,\ldots,\sigma^n\}=\mathcal L(P)$: $\alpha$ is a facet normal of $M(P)$ if and only if the maximizing blocks of all $\sigma\in\mathcal L(P)$ span $\alpha^\perp$.

    \emph{Necessary conditions.} By \ref{thm:fnc:1} $\Rightarrow$ \ref{thm:fnc:2}, every facet normal of $M(P)$ satisfies \ref{thm:fnc:a}--\ref{thm:fnc:c} with at most $|\mathcal L(P)|$ classes. Condition \ref{thm:fnc:c} by itself, however, is vacuous once $|\mathcal L(P)|\geq d-1$, because then each subset $B_j$ can form its own class $\{j\}$; the content lies in \ref{thm:fnc:a}, \ref{thm:fnc:b}, and in the structure of the subsets, which in the proof of \ref{thm:fnc:1} $\Rightarrow$ \ref{thm:fnc:2} are maximizing blocks $B_{k,j}(\alpha)$ of linear extensions $\sigma^k$ of $P$. Each such block is the set of entries of $\sigma^k$ in the consecutive positions $r^k_{j-1}+1,\ldots,r^k_j$ (\cref{subsec:support}), so:
    \begin{enumerate}[label=(\roman*)]
        \item each $B_j$ is \emph{order-convex} in $P$: if $a\prec_P b\prec_P c$ and $a,c\in B_j$\,, then $b\in B_j$\,;
        \item the subsets in one class occur as pairwise disjoint sets of consecutive entries of a single linear extension of $P$.
    \end{enumerate}
    This structure is what the first paper exploits. With $\ell(P)$ the least number of elements whose deletion from $P$ leaves a chain, $M(P)$ is alcoved if and only if $\ell(P)\leq1$ \CP{\CPvertical}, and in this case every primitive facet normal of $M(P)$ is one of $\pm\mathbf{e}_i$ or $\mathbf{e}_i-\mathbf{e}_j$\,, and hence lies in $\{-1,0,1\}^d$ (the converse fails: for the antichain on $[3]$, every primitive facet normal of $M(P)$ lies in $\{-1,0,1\}^3$ by \cref{cor:deltabound}, because $\delta_2=1$, but $(1,1,-1)$ is a facet normal, so $M(P)$ is not alcoved); further, if $\ell(P)\geq1$, the primitive facet normals of $M(P)$ have entries at most $\delta_{\ell(P)}\binom{d-1}{\ell(P)-1}$, which is polynomial in $d$ of degree $\ell(P)-1$ for bounded $\ell(P)$, and this degree is attained \CP{\CPposetheight, \CPgrowth}; and if $\ell(P)$ grows like $d/3$, the entries can be as large as $\mathfrak{B}(d-2)=\Theta(3^{d/3})$ \CP{\CPexp}. The two kinds of constraint behave very differently. For a chain with two isolated elements, $M(P)$ is a sum of $d(d-1)$ staircase simplices, yet its primitive facet normals have entries at most $d-2$ \CP{\CPsharptwo}; by contrast, three unrestricted summands already give entries $\mathrm{Pad}_{d-1}$\,, exponential in $d$ (\cref{thm:padovan}). In the other direction, every facet normal of $M(P)$ is a facet normal of a sum of staircase simplices, so the first paper's \CPcoeffH{} satisfies $H_\ell(d)\leq\mathfrak{h}(d)$, which is at most $\delta_{d-1}$ (\cref{cor:deltabound}) and equals $\operatorname{coeff}(d)$ (\cref{prop:coeff}). For $\ell\geq3$ and small $d$ these bounds are sharper than those of \cite{QL_orderpolytope}: the determinant bound $\delta_6=9$, together with the construction of \CP{\CPlowerfour}, already yields $H_4(7)=9$, and $\operatorname{coeff}(8)=18$ gives $H_3(8)\leq18$ and $H_4(8)\leq18$ \CP{\CPopen}. At the other extreme, the antichain gives $M(P)=U_d$ (\cref{rem:universal}), whose facet normals reach $\operatorname{coeff}(d)$ (\cref{prop:coeff}); so the bounds of the two papers are complementary: the first is governed by the order relations, the second by the number of summands.

    \emph{The converse fails.} For the chain $1\prec_P2\prec_P3$, $M(P)=\Delta_{(1,2,3)}$, and $\alpha=\mathbf{e}_1-\mathbf{e}_3$ satisfies \ref{thm:fnc:2} with $n=0$ via the disjoint subsets $\{1,3\}$ and $\{2\}$, yet it is not a facet normal of $\Delta_{(1,2,3)}$; consistently with (i), $\{1,3\}$ is not order-convex. The obstruction is that the permutations required by \ref{thm:fnc:2} $\Rightarrow$ \ref{thm:fnc:1} need not be linear extensions of $P$.

    \emph{A sufficient condition.} If some choice in the construction of \ref{thm:fnc:2} $\Rightarrow$ \ref{thm:fnc:1} (the order of the blocks, and the order within each part and sign class) produces permutations that all lie in $\mathcal L(P)$, then $\alpha$ is a facet normal of the sum of the corresponding simplices. Discarding repeated permutations (see the comment after \cref{thm:facetnormalcharacterization}) and adding the remaining linear extensions (\cref{lem:persistenceaddition}), $\alpha$ is a facet normal of $M(P)$. \Cref{ex:facetnormalconstruction}(ii) is an instance \textup{(}see also \CP{\CPexample}\textup{)}: its permutations $(1,3,2)$ and $(2,3,1)$ are linear extensions of the antichain on $[3]$, and indeed $(1,1,-1)$ is a facet normal of $M(P)$ for that antichain, in accordance with \CP{\CPvertical} ($\ell(P)=2$, so $M(P)$ is not alcoved).

\backmatter

\section*{Declarations}

\bmhead{Funding}
This work was supported in part by the Office of Naval Research under grant N00014-24-1-2694.

\bmhead{Competing interests}
The authors have no competing interests to declare that are relevant to the content of this article.

\bmhead{Declaration of generative AI and AI-assisted technologies in the manuscript preparation process}
During the preparation of this work, the authors used Claude to assist in developing and checking the work. The authors take full responsibility for the content of the article.

\bmhead{Data availability}
Data sharing is not applicable to this article as no datasets were generated or analyzed during the current study.

\bibliography{staircase}

\end{document}